\documentclass{birkmult}
\usepackage{amsmath}
\usepackage{amsfonts}
\usepackage{amssymb}
\usepackage{graphicx}
\newcommand{\bs}{\boldsymbol}
\newtheorem{thm}{Theorem}[section]
\newtheorem{cor}[thm]{Corollary}
\newtheorem{lem}[thm]{Lemma}
\newtheorem{prop}[thm]{Proposition}
\newtheorem{assum}[thm]{Assumption}
\theoremstyle{definition}
\newtheorem{defn}[thm]{Definition}
\theoremstyle{remark}
\newtheorem{rem}[thm]{Remark}

\numberwithin{equation}{section}
\begin{document}

\title[Spectral Convergence of Large Block-Hankel Gaussian Random Matrices]
{Spectral Convergence of Large Block-Hankel Gaussian Random Matrices}
\author{Philippe Loubaton}
\address{Universit\'e Paris-Est\\
Laboratoire d'Informatique Gaspard Monge, UMR CNRS 8049\\
5 Bd. Descartes, Cit\'e Descartes, Champs sur Marne\\
Marne la Vall\'ee 77454 Cedex 2\\
France}
\email{loubaton@univ-mlv.fr}
\thanks{This work was supported by the Labex B\'ezout under grant ANR-10-LABX-0058}
\author{Xavier Mestre}
\address{Centre Tecnol\`ogic de Telecomunicacions de Catalunya\\
Av. Carl Friedrich Gauss, 7, Parc Mediterrani de la Tecnologia\\
08860 Castelldefels\\
Spain}
\email{xavier.mestre@cttc.cat}
\subjclass{Primary 60B20; Secondary 15B52}
\keywords{Large random matrices, Stieltjes transform of positive matrix-valued measures, Hankel matrices.}
\date{\today}
\dedicatory{Dedicated to Prof. Daniel Alpay on occasion of his 60th birthday}

\begin{abstract}
This paper studies the behaviour of the empirical eigenvalue distribution of 
large random matrices ${\bf W}_N {\bf W}_N^{H}$ where ${\bf W}_N$ is a $ML \times N$ 
matrix, whose $M$ block lines of dimensions $L \times N$ are mutually independent 
Hankel matrices constructed from complex Gaussian correlated stationary random sequences. In the 
asymptotic regime where $M \rightarrow +\infty$, $N \rightarrow +\infty$ and 
$\frac{ML}{N} \rightarrow c > 0$, it is shown using the Stieltjes transform 
approach that the empirical eigenvalue distribution of  ${\bf W}_N {\bf W}_N^{H}$
has a deterministic behaviour which is characterized.  
\end{abstract}
\maketitle

\section{Introduction}

\subsection{The addressed problem and summary of the main results.}

In this paper, we consider a $ML\times N$ block-Hankel matrix $\mathbf{W}_{N}$
composed of $M$ Hankel matrices gathered on top of each other, namely
\[
\mathbf{W}_{N}=\left[
\begin{array}
[c]{ccc}
\mathbf{W}_{1,N}^{T} & \cdots & \mathbf{W}_{M,N}^{T}
\end{array}
\right]  ^{T}.
\]
For each $m=1,\ldots,M$, $\mathbf{W}_{m,N}$ is a Hankel matrix of dimensions
$L\times N$, with $\left(  i,j\right)  $th entry equal to
\[
\left\{  \mathbf{W}_{m,N}\right\}  _{i,j}=w_{m,N}(i+j-1)
\]
for $1\leq i\leq L$, $1\leq j\leq N$, where the random variables
$(w_{m,N}(n))_{m=1,\ldots,M,n=1,\ldots,N+L-1}$ are zero mean complex Gaussian
random variables . The different blocks $\mathbf{W}_{m,N}$ are independent,
but we allow for some time invariant correlation structure within each Hankel
matrix, namely
\[
\mathbb{E}\left[  w_{m,N}(k)w_{m^{\prime},N}^{\ast}(k^{\prime})\right]
=\frac{r_{m}\left(  k-k^{\prime}\right)  }{N}\delta_{m-m^{\prime}}
\]
where $r_{m}\left(  k\right)  $, $k\in\mathbb{Z}$, is a sequence of
correlation coefficients defined as
\[
r_{m}\left(  k\right)  =\int_{0}^{1}\mathcal{S}_{m}\left(  \nu\right)
\mathrm{e}^{2\mathrm{i}\pi\nu k}d\nu
\]
where $(\mathcal{S}_{m})_{m=1,\ldots,M}$ are positive functions. We denote by
$(\hat{\lambda}_{k,N})_{k=1,\ldots,ML}$ the eigenvalues of random matrix
$\mathbf{W}_{N}\mathbf{W}_{N}^{H}$, where $($\textperiodcentered$)^{H}$ stands
for transpose conjugate. The purpose of this paper is to study the asymptotic
properties of the empirical eigenvalue distribution $d\hat{\mu}_{N}
(\lambda)=\frac{1}{ML}\sum_{k=1}^{ML}\delta_{\lambda-\hat{\lambda}_{k,N}}$ of
$\mathbf{W}_{N}\mathbf{W}_{N}^{H}$ when $M\rightarrow+\infty$, $N\rightarrow
+\infty$ and $L$ is such that $c_{N}=\frac{ML}{N}$ converges towards a non
zero constant $c>0$.

It is well established that the asymptotic behaviour of the empirical
eigenvalue distribution of large Hermitian matrices can be evaluated by
studying the behaviour of their Stieltjes transforms. In the context of the
present paper, the Stieltjes transform of $d\hat{\mu}_{N}(\lambda)$ is the
function $q_{N}(z)$ defined on $\mathbb{C}\setminus\mathbb{R}^{+}$ by
\[
q_{N}(z)=\int_{\mathbb{R}^{+}}\frac{1}{\lambda-z}\,d\hat{\mu}_{N}
(\lambda)=\frac{1}{ML}\sum_{k=1}^{ML}\frac{1}{\hat{\lambda}_{k,N}-z}
\]
and also coincides with $q_{N}(z)=\frac{1}{ML}\mathrm{tr}\mathbf{Q}_{N}(z)$
where $\mathbf{Q}_{N}(z)$ is the resolvent of matrix $\mathbf{W}_{N}
\mathbf{W}_{N}^{H}$ defined by
\[
\mathbf{Q}_{N}(z)=\left(  \mathbf{W}_{N}\mathbf{W}_{N}^{H}-z\mathbf{I}\right)
^{-1}.
\]
We denote by $\mathcal{S}_{ML}(\mathbb{R}^{+})$ the set of all $ML\times ML$
matrix valued functions defined on $\mathbb{C}\setminus\mathbb{R}^{+}$ by
\[
\mathcal{S}_{ML}(\mathbb{R}^{+})=\left\{  \int_{\mathbb{R}^{+}}\frac
{1}{\lambda-z}\,d{\boldsymbol\mu}(\lambda)\right\}
\]
where ${\boldsymbol\mu}$ is a positive $ML\times ML$ matrix-valued measure
carried by $\mathbb{R}^{+}$ satisfying ${\boldsymbol\mu}(\mathbb{R}
^{+})=\mathbf{I}_{ML}$. In this paper, we establish that there exists a
function $\mathbf{T}_{N}(z)$ of $\mathcal{S}_{ML}(\mathbb{R}^{+})$, defined as
the unique element of $\mathcal{S}_{ML}(\mathbb{R}^{+})$ satisfying a certain
functional equation, that verifies
\begin{equation}
\frac{1}{ML}\mathrm{tr}\left(  (\mathbf{Q}_{N}(z)-\mathbf{T}_{N}
(z))\mathbf{A}_{N}\right)  \rightarrow0
\label{eq:convergence-normalized-trace}
\end{equation}
almost surely for each $z\in\mathbb{C}\setminus\mathbb{R}^{+}$, where
$(\mathbf{A}_{N})_{N\geq1}$ is an arbitrary sequence of deterministic
$ML\times ML$ matrices satisfying $\sup_{N}\Vert\mathbf{A}_{N}\Vert<+\infty$.
Particularized to the case where $\mathbf{A}_{N}=\mathbf{I}$, this property
implies that
\begin{equation}
q_{N}(z)-t_{N}(z)\rightarrow0 \label{eq:convergence-q-t}
\end{equation}
almost surely for each $z\in\mathbb{C}\setminus\mathbb{R}^{+}$, where
$t_{N}(z)=\frac{1}{ML}\mathrm{tr}(\mathbf{T}_{N}(z))$ is the Stieltjes
transform of the probability measure $\mu_{N}=\frac{1}{ML}\mathrm{tr}
({\boldsymbol\mu}_{N})$. In the present context, this turns out to imply that
almost surely, for each bounded continuous function $\phi$ defined on
$\mathbb{R}^{+}$, it holds that
\begin{equation}
\frac{1}{ML}\sum_{k=1}^{ML}\phi(\hat{\lambda}_{k,N})-\int_{\mathbb{R}^{+}}
\phi(\lambda)\,d\mu_{N}(\lambda)\rightarrow0.
\label{eq:convergence-linear-statistics}
\end{equation}
It is also useful to study the respective rates of convergence towards $0$ of
the variance and of the mean of $\frac{1}{ML}\sum_{k=1}^{ML}\phi(\hat{\lambda
}_{k,N})-\int_{\mathbb{R}^{+}}\phi(\lambda)\,d\mu_{N}(\lambda)$ when $\phi$ is
a smooth function. For this, it appears sufficient to restrict to the case
$\phi(\lambda)=\frac{1}{\lambda-z}$ where $z\in\mathbb{C}\setminus
\mathbb{R}^{+}$, i.e. to study the rate of convergence of $\mathrm{var}
(q_{N}(z))=\mathbb{E}|q_{N}(z)-\mathbb{E}(q_{N}(z))|^{2}$ and of
$\mathbb{E}(q_{N}(z))-\int_{\mathbb{R}^{+}}\phi(\lambda)\,d\mu_{N}(\lambda)$.
More generally, if $(\mathbf{A}_{N})_{N\geq1}$ is any sequence of
deterministic $ML\times ML$ matrices satisfying $\sup_{N}\Vert\mathbf{A}
_{N}\Vert<+\infty$, we establish that
\begin{equation}
\mathrm{var}\left[  \frac{1}{ML}\mathrm{tr}\left(  \mathbf{Q}_{N}
(z)\mathbf{A}_{N}\right)  \right]  =\mathcal{O}(\frac{1}{MN})
\label{eq:convergence-rate-variance}
\end{equation}
and that, provided $\frac{L^{3/2}}{MN}\rightarrow0$,
\begin{equation}
\frac{1}{ML}\mathrm{tr}\left(  \left(  \mathbb{E}(\mathbf{Q}_{N}
(z))-\mathbf{T}_{N}(z)\right)  \mathbf{A}_{N}\right)  =\mathcal{O}(\frac
{L}{MN}) \label{eq:convergence-rate-biais}
\end{equation}
for each $z\in\mathbb{C}\setminus\mathbb{R}^{+}$.

In this paper, we concentrate on the characterization of the asymptotic
behaviour of the terms $\frac{1}{ML}\mathrm{tr}\left(  \mathbf{Q}
_{N}(z)\mathbf{A}_{N}\right)  $, and do not discuss on the behaviour of
general linear statistics $\frac{1}{ML}\sum_{k=1}^{ML}\phi(\hat{\lambda}
_{k,N})$ of the eigenvalues. In order to establish our results, we follow the
general approach introduced in \cite{pastur-simple} and developed in more
general contexts in \cite{pastur-shcherbina-book}. This approach takes benefit
of the Gaussianity of the random variables $w_{m}(n)$, and use the
Poincar\'{e}-Nash inequality to evaluate the variance of various terms and the
integration by parts formula to evaluate approximations of matrix
$\mathbb{E}(\mathbf{Q}_{N}(z))$. \newline

Apart large random matrix methods, the properties of Stieltjes transforms of
positive matrix valued measures play a crucial role in this paper. The first
author (in the alphabetic order) of this paper had the chance to learn these tools from Prof. D. Alpay
at the occasion of past collaborations. The authors are thus delighted to
dedicate this paper to Prof. D. Alpay on occasion of his 60th birthday.

\subsection{Motivations}

The present paper is motivated by the problem of testing whether $M$ complex
Gaussian zero mean times series $(x_{m}(n))_{n\in\mathbb{Z}}$ are mutually
independent. For each $m=1,\ldots,M$, $x_{m}$ is observed from time $n=1$ to
$n=N$, and a relevant statistics depending on $((x_{m}(n))_{n=1,\ldots
,N})_{m=1,\ldots,M}$ has to be designed and studied. A reasonable approach can
be drawn by noting that if the $M$ time series are independent, then, for each
integer $L$, the covariance matrix $\mathbf{R}_{L}$ of $ML$-dimensional random
vector $\mathbf{x}_{L}(n)=(x_{1}(n),\ldots,x_{1}(n+L-1),x_{2}(n),\ldots
,x_{2}(n+L-1),\ldots,x_{M}(n),\ldots,x_{M}(n+L-1))^{T}$ is block diagonal, a
property implying that
\begin{equation}
\kappa_{N}=\frac{1}{ML}\left(  \log\mathrm{det}(\mathbf{R}_{L})-\sum_{m=1}
^{M}\log\mathrm{det}(\mathbf{R}_{L}^{m,m})\right)  =0 \label{eq:test-simple}
\end{equation}
where $\mathbf{R}_{L}^{m,m}$ represents the $m$-th $L\times L$ diagonal block
of $\mathbf{R}_{L}$. Therefore, it seems relevant to approximate matrix
$\mathbf{R}_{L}$ by the standard estimator $\hat{\mathbf{R}}_{L}$ defined by
\[
\hat{\mathbf{R}}_{L}=\frac{1}{ML}\sum_{n=1}^{N}\mathbf{x}_{L}(n)\left(
\mathbf{x}_{L}(n)\right)  ^{H}
\]
so that we can evaluate the term $\hat{\kappa}_{N}$, obtained by replacing
$\mathbf{R}_{L}$ by $\hat{\mathbf{R}}_{L}$ in (\ref{eq:test-simple}), and
compare it to $0$. The present paper is motivated by the study of this
particular test under the hypothesis that the series $(x_{m})_{m=1,\ldots,M}$
are uncorrelated and assuming that $M$ and $N$ are both large. 
In this context, a crucial problem is to choose parameter $L$. On one hand,
$L$ should be chosen in such a way that $ML/N<<1$ in order to make the
estimation error $\Vert\hat{\mathbf{R}}_{L}-\mathbf{R}_{L}\Vert$ reasonably
low, and thus $\hat{\kappa}_{N}$ close to $0$ under the uncorrelation
hypothesis. On the other hand, choosing a small value for $L$ is not
satisfying because comparing $\hat{\kappa}_{N}$ to $0$ allows to test that
$\mathbb{E}(x_{m}(l)x_{m^{\prime}}^{\ast}(l^{\prime}))=0$ for each
$(m,m^{\prime})$ only for $l,l^{\prime}\in\lbrack0,\ldots,L-1]$, a property
that does not imply formally that the time series $x_{m}$ and $x_{m^{\prime}}$
are independent. Therefore, choosing $L$ as large as possible is relevant. In
this case, the ratio $\frac{ML}{N}$ may no longer be very small, and
$\hat{\kappa}_{N}$ may not converge towards $0$. It is thus of fundamental
interest to evaluate the behaviour of $\hat{\kappa}_{N}$ when $M$ and $N$ are
large and that $\frac{ML}{N}$ is not negligible. This question is connected to
the problem addressed in the present paper because the following results
potentially allow to establish that $\frac{1}{ML}\log\mathrm{det}
(\hat{\mathbf{R}}_{L})$ has a deterministic behaviour that can be
characterized. This term can indeed be written as
\[
\frac{1}{ML}\log\mathrm{det}(\hat{\mathbf{R}}_{L})=\frac{1}{ML}\sum_{k=1}
^{ML}\phi(\hat{\lambda}_{k,N})
\]
where $(\hat{\lambda}_{k,N})_{k=1,\ldots,ML}$ are the eigenvalues of
$\hat{\mathbf{R}}_{L}$ and where $\phi(\lambda)=\log\lambda$. Moreover, if we
denote by $w_{m,N}(n)$ the random variable $w_{m,N}(n)=\frac{1}{\sqrt{N}}
x_{m}(n)$, then it is easily seen that matrix $\hat{\mathbf{R}}_{L}$ coincides
with matrix $\mathbf{W}_{N}\mathbf{W}_{N}^{H}$ where $\mathbf{W}_{N}$ is
constructed from the $w_{m,N}(n)$ as above, up to end effects (because matrix
$\mathbf{W}_{N}$ depends on random variables $(w_{m}(N+l))_{m=1,\ldots
,M,l=1,\ldots,L-1}$ while matrix $\hat{\mathbf{R}}_{L}$ does not depend on
these entries). However, these end effects can be shown to be negligible.
Therefore, the asymptotic behaviour of $\frac{1}{ML}\log\mathrm{det}
(\hat{\mathbf{R}}_{L})$ appears to be a consequence of the results of the
present paper. We finally mention that under some reasonable assumptions,
$\Vert\hat{\mathbf{R}}_{L}^{m,m}-\mathbf{R}_{L}^{m,m}\Vert\rightarrow0$ so
that it holds that
\[
\frac{1}{ML}\sum_{m=1}^{M}\log\mathrm{det}(\hat{\mathbf{R}}_{L}^{m,m}
)-\frac{1}{ML}\sum_{m=1}^{M}\log\mathrm{det}(\mathbf{R}_{L}^{m,m}
)\rightarrow0
\]
In summary, the asymptotic behaviour of $\hat{\kappa}_{N}$ appears to be a
consequence the study of the empirical eigenvalue distribution of
$\mathbf{W}_{N}\mathbf{W}_{N}^{H}$ in the asymptotic regime where
$M,N\rightarrow+\infty$ in such a way that $\frac{ML}{N}$ converges towards a
non zero constant.

\subsection{On the literature.}

The study of the asymptotic behaviour of large random Gram matrices has a long
history. Since the pionneering work of Marcenko-Pastur in 1967
(\cite{marchenko67}), a number of random matrix models have been considered
(see e.g. \cite{bai-silverstein-book}, \cite{pastur-shcherbina-book} and the
references therein). In the following, we mention the previous works that are
connected to the present paper. As matrix $\mathbf{W}_{N}$ is a block line
matrix with $L\times N$ blocks, we mention the works of Girko
(\cite{girko-book}, chapter 16) as well as \cite{bryc-speicher-2006} devoted
to the case where the blocks are i.i.d. We however mention that
\cite{girko-book} and \cite{bryc-speicher-2006} did not characterize the rates
of convergence and that the techniques used in these works do not allow to
address the case where $L\rightarrow+\infty$. The works devoted to Hankel
matrices are also relevant. \cite{basak-bose-sen-2012} addressed the case
where $M=1$ and $N,L\rightarrow+\infty$ at the same rate, except that in
\cite{basak-bose-sen-2012}, the random variables $w_{1,N}(n)$ are forced to
$0$ for $N<n\leq N+L-1$. Using the moments method, \cite{basak-bose-sen-2012}
showed that the empirical eigenvalue distribution of $\mathbf{W}_{N}
\mathbf{W}_{N}^{\ast}$ converges towards a non compactly supported limit
distribution. The random matrix model considered in \cite{loubaton16} is
similar to matrix $\mathbf{W}_{N}$ of the present paper, except that in
\cite{loubaton16}, for each $m=1,\ldots,M$, the random variables $(w_{m}(n))$
are uncorrelated, i.e. the spectral densities $(\mathcal{S}_{m}(\nu
))_{m=1,\ldots,M}$ are reduced to $\mathcal{S}_{m}(\nu)=1$ for each $\nu$.
Using the Poincar\'{e}-Nash inequality and the integration by parts formula,
\cite{loubaton16} studied the asymptotic behaviour of the empirical eigenvalue
distribution $\hat{\mu}_{N}$ in the asymptotic regime $M,N\rightarrow+\infty$
and $\frac{ML}{N}\rightarrow c$, $c>0$. It was established that function
$t_{N}(z)$ defined in (\ref{eq:convergence-q-t}) coincides with the Stieltjes
transform of the Marcenko-Pastur distribution with parameter $c_{N}$, so that
$\hat{\mu}_{N}$ converges weakly almost surely towards the Marcenko pastur
distribution. The rates of convergence of the variance and of the expectation
of $q_{N}(z)-t_{N}(z)$ are both characterized. Finally, \cite{loubaton16}
proved that provided $L=\mathcal{O}(N^{\alpha})$ with $\alpha<2/3$, then the
extreme non zero eigenvalues of $\mathbf{W}_{N}\mathbf{W}_{N}^{\ast}$ converge
almost surely towards the end points of the support of the Marcenko-Pastur
distribution. Therefore, the present paper is a partial generalization of
\cite{loubaton16}.

\subsection{Assumptions, general notations, and background on Stieltjes
transforms of positive matrix valued measures.}

\textbf{Assumptions on $L,M,N$}

\begin{assum}
\label{as:standard}

\begin{itemize}
\item All along the paper, we assume that $L,M,N$ satisfy $M\rightarrow
+\infty,N\rightarrow+\infty$ in such a way that $c_{N}=\frac{ML}{N}\rightarrow
c$, where $0<c<+\infty$. In order to shorten the notations, $N\rightarrow
+\infty$ should be understood as the above asymptotic regime.

\item In section \ref{sec:convergence-R-T}, $L,M,N$ also satisfy
$\frac{L^{3/2}}{MN} \rightarrow0$ or equivalently $\frac{L}{M^{4}}
\rightarrow0$.
\end{itemize}
\end{assum}

\textbf{Assumptions on sequences $(r_{m})_{m=1,\ldots,M}$ and spectral
densities $(\mathcal{S}_{m})_{m=1,\ldots,M}$}. We assume that sequences
$(r_{m})_{m=1,\ldots,M}$ satisfy the condition
\begin{equation}
\sup_{M}\sum_{n\in\mathbb{Z}}\left(  \frac{1}{M}\sum_{m=1}^{M}|r_{m}
(n)|^{2}\right)  ^{1/2}<+\infty.\label{eq:condition-rm}
\end{equation}
As it holds that $\left(  \frac{1}{M}\sum_{m=1}^{M}|r_{m}(n)|\right)  ^{2}
\leq\frac{1}{M}\sum_{m=1}^{M}|r_{m}(n)|^{2}$, condition (\ref{eq:condition-rm}
) implies that
\[
\sup_{M}\frac{1}{M}\sum_{m=1}^{M}\sum_{n\in\mathbb{Z}}|r_{m}(n)|<+\infty
\]
and that for each $m$, $\sum_{n\in\mathbb{Z}}|r_{m}(n)|<+\infty$. Therefore,
each spectral density $\mathcal{S}_{m}$ is continuous. We also assume that
\begin{align}
\sup_{M} \sup_{m=1,\ldots,M}\max_{\nu\in\lbrack0,1]}\mathcal{S}_{m}(\nu) &
<+\infty\label{eq:upperbound-S}\\
\inf_{M} \inf_{m=1,\ldots,M}\min_{\nu\in\lbrack0,1]}\mathcal{S}_{m}(\nu) &
>0.\label{eq:lowerbound-S}
\end{align}
In the following, we denote by $\mathcal{R}(k)$ the diagonal matrix
\begin{equation}
\mathcal{R}(k)=\mathrm{diag}(r_{1}(k),\ldots,r_{M}
(k)).\label{eq:def-mathcal-R}
\end{equation}

\textbf{General notations.} \newline

In the following, we will often drop the index $N$, and will denote
$\mathbf{W}_{N},\mathbf{Q}_{N},\ldots$ by $\mathbf{W}, \mathbf{Q},\ldots$ in
order to simplify the notations. The $N$ columns of matrix $\mathbf{W}$ are
denoted $(\mathbf{w}_{j})_{j=1,\ldots,N}$. For $1\leq l\leq L$, $1\leq m\leq
M$, and $1\leq j\leq N$, $\mathbf{W}_{i,j}^{m}$ represents the entry $\left(
i+(m-1)L,j\right)  $ of matrix $\mathbf{W}$.

If $\mathbf{A}$ is a $ML \times ML$ matrix, we denote by $\mathbf{A}
^{m_{1},m_{2}}_{i_{1},i_{2}}$ the entry $(i_{1} + (m_{1}-1)L, i_{2} +
(m_{2}-1)L)$ of matrix $\mathbf{A}$, while $\mathbf{A}^{m_{1},m_{2}}$
represents the $L \times L$ matrix $(\mathbf{A}^{m_{1},m_{2}}_{i_1,i_2})_{1 \leq
(i_{1},i_{2}) \leq L}$.

$\mathbb{C}^{+}$ denotes the set of complex numbers with strictly positive
imaginary parts. The conjuguate of a complex number $z$ is denoted $z^{\ast}$.
If $z\in\mathbb{C}\setminus\mathbb{R}^{+}$, we denote by $\delta_{z}$ the
term
\begin{equation}
\delta_{z}=\mathrm{dist}(z,\mathbb{R}^{+}).\label{eq:def-delta_z}
\end{equation}
The conjugate transpose of a matrix $\mathbf{A}$ is denoted $\mathbf{A}^{H}$
while the conjugate of $\mathbf{A}$ (i.e. the matrix whose entries are the
conjugates of the entries of $\mathbf{A}$) is denoted $\mathbf{A}^{\ast}$.

$\| \mathbf{A} \|$ represents the spectral norm of matrix $\mathbf{A}$. If
$\mathbf{A}$ and $\mathbf{B}$ are 2 matrices, $\mathbf{A} \otimes\mathbf{B}$
represents the Kronecker product of $\mathbf{A}$ and $\mathbf{B}$, i.e. the
block matrix whose block $(i,j)$ is $\mathbf{A}_{i,j} \, \mathbf{B}$. If
$\mathbf{A}$ is a square matrix, $\mathrm{Im}(\mathbf{A})$ and $\mathrm{Re}
(\mathbf{A})$ represent the Hermitian matrices
\[
\mathrm{Im}(\mathbf{A}) = \frac{\mathbf{A} - \mathbf{A}^{H}}{2i}, \;
\mathrm{Re}(\mathbf{A}) = \frac{\mathbf{A} + \mathbf{A}^{H}}{2}
\]

If $(\mathbf{A}_{N})_{N \geq1}$ (resp. $(\mathbf{b}_{N})_{N \geq1}$) is a
sequence of matrices (resp. vectors) whose dimensions increase with $N$,
$(\mathbf{A}_{N})_{N \geq1}$ (resp. $(\mathbf{b}_{N})_{N \geq1}$) is said to
be uniformly bounded if $\sup_{N \geq1} \| \mathbf{A}_{N} \| < +\infty$ (resp.
$\sup_{N \geq1} \| \mathbf{b}_{N} \| < +\infty$).

If $\nu\in\lbrack0,1]$ and if $R$ is an integer, we denote by $\mathbf{d}
_{R}(\nu)$ the $R$--dimensional vector $\mathbf{d}_{R}(\nu)=(1,e^{2\mathrm{i}
\pi\nu},\ldots,e^{2\mathrm{i}\pi(R-1)\nu})^{T}$, and by $\mathbf{a}_{L}(\nu)$
the vector $\mathbf{a}_{L}(\nu)=\frac{1}{\sqrt{R}}\,\mathbf{d}_{R}(\nu)$.

If $x$ is a complex-valued random variable, the variance of $x$, denoted by
$\mathrm{Var}(x)$, is defined by
\[
\mathrm{Var}(x) = \mathbb{E}(|x|^{2}) - \left|  \mathbb{E}(x) \right|  ^{2}
\]
The zero-mean random variable $x - \mathbb{E}(x)$ is denoted $x^{\circ}$. \newline

\textbf{Nice constants and nice polynomials}

\begin{defn}
\label{def:nice}
A nice constant is a positive
constant independent of the dimensions $L,M,N$ and complex variable $z$. A
nice polynomial is a polynomial whose degree is independent from $L,M,N$, and
whose coefficients are nice constants. 
\end{defn}
In the following, $P_{1}$ and $P_{2}$
will represent generic nice polynomials whose values may change from one line
to another, and $C(z)$ is a generic term of the form $C(z) = P_{1}(|z|)
P_{2}(1/\delta_{z})$. \\

\textbf{Background on Stieltjes transforms of positive matrix valued
measures.} In the following, we denote by $\mathcal{S}_{K}(\mathbb{R}^{+})$
the set of all Stieltjes transforms of $K \times K$ positive matrix-valued
measures ${\boldsymbol \mu}$ carried by $\mathbb{R}^{+}$ verifying
${\boldsymbol \mu}_{K}(\mathbb{R}^{+}) = \mathbf{I}_{K}$. The elements of the
class $\mathcal{S}_{K}(\mathbb{R}^{+})$ satisfy the following properties:

\begin{prop}
\label{prop:class-S} Consider an element $\mathbf{S}(z) = \int_{\mathbb{R}
^{+}} \frac{d \, {\boldsymbol \mu}(\lambda)}{\lambda- z}$ of $\mathcal{S}
_{K}(\mathbb{R}^{+})$. Then, the following properties hold true:

\begin{itemize}
\item (i) $\mathbf{S}$ is analytic on $\mathbb{C} \setminus\mathbb{R}^{+}$

\item (ii) $\mathrm{Im}(\mathbf{S}(z)) \geq0$ and $\mathrm{Im}(z \,
\mathbf{S}(z)) \geq0$ if $z \in\mathbb{C}^{+}$

\item (iii) $\lim_{y\rightarrow+\infty}-\mathrm{i}y\mathbf{S}(\mathrm{i}
y)=\mathbf{I}_{K}$

\item (iv) $\mathbf{S}(z)\mathbf{S}^{H}(z)\leq\frac{\mathbf{I}_{K}}{\delta
_{z}^{2}}$ for each $z\in\mathbb{C}\setminus\mathbb{R}^{+}$

\item (v) $\int_{\mathbf{R}^{+}}\lambda\,d{\boldsymbol\mu}(\lambda
)=\lim_{y\rightarrow+\infty}\mathrm{Re}\left(  -\mathrm{i}y(\mathbf{I}
_{K}+\mathrm{i}y\mathbf{S}(iy)\right)  $
\end{itemize}

Conversely, if a function $\mathbf{S}(z)$ satisfy properties (i), (ii), (iii),
then $\mathbf{S}(z)\in\mathcal{S}_{K}(\mathbb{R}^{+})$
\end{prop}

While you have not been able to find a paper in which this result is proved,
it has been well known for a long time (see however
\cite{hachem-loubaton-najim-aap-2007} for more details on (i), (ii), (iii),
(v)), as well as Theorem 3 of \cite{alpay-tsekanovskii} from which (iv) follows immediately) . 
\section{Toeplitzification operators.}

In the following derivations, it will be useful to consider the following
Toeplitzification operators, which inherently depend on the correlation
function $r_{m}\left(  \text{\textperiodcentered}\right)  $. Let
$\mathbf{J}_{K}$ denote the $K\times K$ shift matrix with ones in the first
upper diagonal and zeros elsewhere, namely $\left\{  \mathbf{J}_{K}\right\}
_{i,j}=\delta_{j-i-1}$, and let $\mathbf{J}_{K}^{-1}$ denote its transpose.
For a given squared matrix $\mathbf{M}$ with dimensions $R\times R$, we define
$\Psi_{K}^{(m)}\left(  \mathbf{M}\right)  $ as an $K\times K$ Toeplitz matrix
with $(i,j)$th entry equal to
\begin{equation}
\left\{  \Psi_{K}^{(m)}\left(  \mathbf{M}\right)  \right\}  _{i,j}
=\sum_{l=-R+1}^{R-1}r_{m}\left(  i-j-l\right)  \tau\left(  \mathbf{M}\right)
\left(  l\right)  \label{eq:definition_operator_Psi}
\end{equation}
or, alternatively, as the matrix
\begin{equation}
\Psi_{K}^{(m)}\left(  \mathbf{M}\right)  =\sum_{n=-K+1}^{K-1}\left(
\sum_{l=-R+1}^{R-1} r_{m}\left(  n-l\right)  \tau\left(
\mathbf{M}\right)  \left(  l\right)  \right)  \mathbf{J}_{K}^{-n}
\label{eq:definition_operator_Psi2}
\end{equation}
where the sequence $\tau\left(  \mathbf{M}\right)  \left(  l\right)  $,
$-R<l<R$, is defined as
\begin{equation}
\tau\left(  \mathbf{M}\right)  \left(  l\right)  =\frac{1}{R}\mathrm{tr}
\left[  \mathbf{MJ}_{R}^{l}\right]  .\label{eq:definition_tau}
\end{equation}

We can express this operator more compactly using frequency notation, namely
\begin{align*}
\Psi_{K}^{(m)}\left(  \mathbf{M}\right)   &  =\sum_{n=-K+1}^{K-1}\left(
\int_{0}^{1}\mathcal{S}_{m}\left(  \nu\right)  \mathbf{a}_{R}^{H}\left(
\nu\right)  \mathbf{Ma}_{R}\left(  \nu\right)  \mathrm{e}^{2\pi\mathrm{i}\nu
n}d\nu\right)  \mathbf{J}_{K}^{-n}\\
&  =\int_{0}^{1}\mathcal{S}_{m}\left(  \nu\right)  \mathbf{a}_{R}^{H}\left(
\nu\right)  \mathbf{Ma}_{R}\left(  \nu\right)  \mathbf{d}_{K}\left(
\nu\right)  \mathbf{d}_{K}^{H}\left(  \nu\right)  d\nu
\end{align*}
where $\mathbf{a}_{R}\left(  \nu\right)  =\mathbf{d}_{R}\left(  \nu\right)
/\sqrt{R}$ and $\mathbf{d}_{R}\left(  \nu\right)  =\left(  1,\mathrm{e}
^{2\pi\mathrm{i}\nu},\ldots,\mathrm{e}^{2\pi\mathrm{i}\left(  R-1\right)  \nu
}\right)  ^{T}$. In particular, when $r_{m}\left(  k\right)  =\sigma^{2}
\delta_{k}$(white observations), we have
\begin{align*}
\Psi_{K}^{(m)}\left(  \mathbf{M}\right)   &  =\sigma^{2}\sum_{n=-K+1}
^{K-1}\frac{1}{R}\mathrm{tr}\left[  \mathbf{MJ}_{R}^{n}\right]  \mathbf{J}
_{K}^{-n}=\sigma^{2}\sum_{n=-K+1}^{K-1}\tau\left(  \mathbf{M}\right)
\mathbf{J}_{K}^{-n}\\
&  =\sigma^{2}\mathcal{T}_{K,\min\left(  R,K\right)  }(\mathbf{M})
\end{align*}
where $\mathcal{T}_{K,R}(\mathbf{X})$ is the classical Toeplitzation operator
in \cite{loubaton16}. The following properties are easily checked.

\begin{itemize}
\item Given a square matrix $\mathbf{A}$ of dimension $K\times K$ and a square
matrix$\ \mathbf{B}$ of dimension $R\times R$, we can write
\begin{equation}
\frac{1}{K}\mathrm{tr}\left[  \mathbf{A}\Psi_{K}^{(m)}\left(  \mathbf{B}
\right)  \right]  =\int_{0}^{1}\mathcal{S}_{m}\left(  \nu\right)
\mathbf{\mathbf{a}}_{K}^{H}\left(  \nu\right)  \mathbf{A\mathbf{a}}_{K}\left(
\nu\right)  \mathbf{a}_{R}^{H}\left(  \nu\right)  \mathbf{Ba}_{R}\left(
\nu\right)  d\nu=\frac{1}{R}\mathrm{tr}\left[  \Psi_{R}^{(m)}\left(
\mathbf{A}\right)  \mathbf{B}\right]  \label{eq:property_commutative}
\end{equation}

\item Given the square matrices $\mathbf{B,C}$ (of dimension $K\times K$) and
$\mathbf{D},\mathbf{E}$ (of dimension $R\times R$), we have
\[
\frac{1}{K}\mathrm{tr}\left[  \mathbf{B}\Psi_{K}^{(m)}\left(  \mathbf{D\Psi
}_{R}^{(m)}\left(  \mathbf{C}\right)  \mathbf{E}\right)  \right]  =\frac{1}
{K}\mathrm{tr}\left[  \mathbf{C\Psi}_{K}^{(m)}\left(  \mathbf{D}\Psi_{R}
^{(m)}\left(  \mathbf{B}\right)  \mathbf{E}\right)  \right]  .
\]

\item Given a square matrix $\mathbf{M}$ and a positive integer $K$, we have
\[
\left\Vert \Psi_{K}^{(m)}\left(  \mathbf{M}\right)  \right\Vert \leq\sup
_{\nu\in\lbrack0,1]}\left\vert \mathcal{S}_{m}\left(  \nu\right)
\mathbf{a}_{R}^{H}\left(  \nu\right)  \mathbf{Ma}_{R}\left(  \nu\right)
\right\vert \leq\sup_{\nu\in\lbrack0,1]}\left\vert \mathcal{S}_{m}\left(
\nu\right)  \right\vert \left\Vert \mathbf{M}\right\Vert .
\]

\item Given a square positive definite matrix $\mathbf{M}$ and a positive
integer $K$, condition (\ref{eq:lowerbound-S}) implies that 
\begin{equation}
\Psi_{K}^{(m)}\left(  \mathbf{M}\right)  >0. \label{eq:property_positive}
\end{equation}

\end{itemize}

We define here two other operators that will be used throughout the paper,
which respectively operate on $N\times N$ and $ML\times ML$ matrices. In order
to keep the notation as simple as possible, we will drop the dimensions in the
notation of these operators.

\begin{itemize}
\item Consider an $N\times N$ matrix $\mathbf{M}$. We define $\Psi\left(
\mathbf{M}\right)  $ as an $ML\times ML$ block diagonal matrix with $m$th
diagonal block given by $\Psi_{L}^{(m)}\left(  \mathbf{M}\right)  $.

\item Consider an $ML\times ML$ matrix $\mathbf{M}$, and let $\mathbf{M}
^{m,m}$ denote its $m$th $L\times L$ diagonal block. We define $\overline
{\Psi}\left(  \mathbf{M}\right)  $ as the $N\times N$ matrix given by
\begin{equation}
\overline{\Psi}\left(  \mathbf{M}\right)  =\frac{1}{M}\sum_{m=1}^{M}\Psi
_{N}^{(m)}\left(  \mathbf{M}^{m,m}\right)  .\label{eq:def_Phi_average}
\end{equation}
$\overline{\Psi}\left(  \mathbf{M}\right)  $ can also be expressed as
\begin{equation}
\overline{\Psi}\left(  \mathbf{M}\right)  =\sum_{n=-(N-1)}^{N-1}
\sum_{l=-(L-1)}^{L-1}\tau^{(M)}\left(  \mathbf{M}(\mathcal{R}(n-l)\otimes
\mathbf{I}_{L})\right)  (l)\;\mathbf{J}_{N}^{-n}\label{eq:expre-1-Psibar}
\end{equation}
where $\tau^{(M)}(\mathbf{A})(l)$ is defined for any $ML\times ML$ matrix
$\mathbf{A}$ by
\[
\tau^{(M)}(\mathbf{A})(l)=\frac{1}{ML}\mathrm{tr}\left(  \mathbf{A}
(\mathbf{I}_{M}\otimes\mathbf{J}_{L}^{l})\right)  =\frac{1}{L}\mathrm{tr}
\left[  \left(  \frac{1}{M}\sum_{m=1}^{M}\mathbf{A}^{m,m}\right)
\mathbf{J}_{L}^{l}\right]
\]
and where $\mathcal{R}(m)$ is defined in (\ref{eq:def-mathcal-R}). Note also
that$\overline{\text{ }\Psi}\left(  \mathbf{M}\right)  $ can alternatively be
written as
\[
\overline{\Psi}\left(  \mathbf{M}\right)  =\frac{1}{M}\sum_{m=1}^{M}\int
_{0}^{1}\mathcal{S}_{m}\left(  \nu\right)  \mathbf{a}_{L}^{H}\left(
\nu\right)  \mathbf{M}^{m,m}\mathbf{a}_{L}\left(  \nu\right)  \mathbf{d}
_{N}\left(  \nu\right)  \mathbf{d}_{N}^{H}\left(  \nu\right)  d\nu.
\]

\end{itemize}

Given these two new operators, and if $\mathbf{A}$ and $\mathbf{B}$ are
$ML\times ML$ and $N\times N$ matrices, we see directly from
(\ref{eq:property_commutative}) that
\begin{equation}
\frac{1}{N}\mathrm{tr}\left[  \overline{\Psi}\left(  \mathbf{A}\right)
\mathbf{B}\right]  =\frac{1}{ML}\mathrm{tr}\left[  \mathbf{A}\Psi\left(
\mathbf{B}\right)  \right]  . \label{eq:transpose_ops_gen}
\end{equation}

\section{Variance evaluations}

\label{sec:variance-evaluations}

In this section, we provide some estimates on the variance of certain
quantities that depend on the resolvent $\mathbf{Q}(z)=\left(  \mathbf{WW}
^{H}-z\mathbf{I}_{ML}\right)  ^{-1}$and co-resolvent $\widetilde{\mathbf{Q}
}(z)=\left(  \mathbf{W}^{H}\mathbf{W}-z\mathbf{I}_{N}\right)  ^{-1}$. We
express the result in the following lemma.

\begin{lem}
\label{le:variance-trace}
Let $(\mathbf{A}_N)_{N \geq 1}$ be a sequence  of deterministic  $ML\times ML$ matrices and $(\mathbf{G}_N)_{N \geq 1}$  a 
uniformly bounded sequence of deterministic  $N\times N$ matrices. Then
\begin{align}
\mathrm{Var}\left(  \frac{1}{ML}\mathrm{tr} {\bf A}_N {\bf Q}(z)\right)   &
\leq\frac{C(z)}{MN}\frac{1}{ML}\mathrm{tr}({\bf A}_N {\bf A}_N^{H})
\label{eq:var_norm_trace}\\
\mathrm{Var}\left(  \frac{1}{ML}\mathrm{tr}{\bf A}_N {\bf Q}(z)\mathbf{WGW}
^{H}\right)   &  \leq\frac{C(z)}{MN}\frac{1}{ML}\mathrm{tr}({\bf A}_N {\bf A}_N
^{H})\label{eq:var_trace_complicated_term}
\end{align}
where $C(z)=P_{1}(\left\vert z\right\vert )P_{2}(1/\delta_{z})$ for two nice
polynomials $P_{1}$, $P_{2}$ (see Definition \ref{def:nice}). 
\end{lem}

We devote the rest of this section to proving of this result. In order to short the notations, matrices 
${\bf A}_N $ and ${\bf G}_N$ will be denoted by ${\bf A}$ and ${\bf G}$. We will be using
the Poincar\'{e}-Nash inequality (\cite{chernoff-annals-proba-1981}, \cite{chen-jmva-1982}), which, in the present context, can be formulated as follows ( \cite{pastur-shcherbina-book, hachem08}).

\begin{lem}
Let $\xi=\xi\left(  \mathbf{W,W}^{\ast}\right)  $ denote a $\mathcal{C}^{1}$
complex function such that both itself and its derivatives are polynomically
bounded. Under the above assumptions, we can write
\begin{align*}
\mathrm{Var}\xi &  \leq\mathbb{E}\sum_{m,i_{1},i_{2},j_{1},j_{2}}\left(
\frac{\partial\xi}{\partial\left(  \mathbf{W}_{i_{1},j_{1}}^{m}\right)
^{\ast}}\right)  ^{\ast}\mathbb{E}\left[  \mathbf{W}_{i_{1},j_{1}}^{m}\left(
\mathbf{W}_{i_{2},j_{2}}^{m}\right)  ^{\ast}\right]  \frac{\partial\xi
}{\partial\left(  \mathbf{W}_{i_{2},j_{2}}^{m}\right)  ^{\ast}}\\
&  +\mathbb{E}\sum_{m,i_{1},i_{2},j_{1},j_{2}}\frac{\partial\xi}
{\partial\left(  \mathbf{W}_{i_{1},j_{1}}^{m}\right)  }\mathbb{E}\left[
\mathbf{W}_{i_{1},j_{1}}^{m}\left(  \mathbf{W}_{i_{2},j_{2}}^{m}\right)
^{\ast}\right]  \left(  \frac{\partial\xi}{\partial\mathbf{W}_{i_{2},j_{2}
}^{m}}\right)  ^{\ast}
\end{align*}
where $\mathbf{W}_{i,j}^{m}$ is the $((m-1)L+i,j)$th entry of $\mathbf{W}$.
\end{lem}

We just check that the first term, denoted $\beta$, on the right hand side of
the upper bound of $\mathrm{Var}\xi$ is in accordance with the results claimed
in Lemma \ref{le:variance-trace} for $\xi=\frac{1}{ML}\mathrm{Tr}
(\mathbf{AQ}(z))$ and $\xi=\frac{1}{ML}\mathrm{Tr}(\mathbf{AQ}(z)\mathbf{WGW}
^{H}))$. For this, we establish that it is possible to be back to the case
where the spectral densities $(\mathcal{S}_{m}(\nu))_{m=1,\ldots,M}$ all
coincide with $1$ which is covered by the results of \cite{loubaton16}. More
precisely, given the Hankel structure of the matrices $\mathbf{W}_{m}$, we can
state that
\begin{equation}
\mathbb{E}\left[  \mathbf{W}_{i_{1},j_{1}}^{m}\left(  \mathbf{W}_{i_{2},j_{2}
}^{m}\right)  ^{\ast}\right]  =\frac{1}{N}r_{m}\left(  i_{1}-i_{2}+j_{1}
-j_{2}\right)  .\label{eq:correlationWW}
\end{equation}
Using that $r_{m}(i_{1}-i_{2}+j_{1}-j_{2})=\int_{0}^{1}e^{2\pi\mathrm{i}
(i_{1}-i_{2}+j_{1}-j_{2})\nu}\mathcal{S}_{m}(\nu)\,d\nu$, we obtain
immediately that $\beta$ can be written as $\beta = \mathbb{E}(\alpha)$ 
where $\alpha$ is defined by 
\[
\alpha =  \frac{1}{N}\int_{0}^{1}\sum_{m=1}^{M}\mathcal{S}
_{m}(\nu)\,\left\vert \sum_{i_{2},j_{2}}\frac{\partial\xi}{\partial\left(
\mathbf{W}_{i_{2},j_{2}}^{m}\right)  ^{\ast}}\,e^{-2\pi\mathrm{i}(i_{2}
+j_{2})\nu}\right\vert ^{2}\,d\nu 
\]
Thus, (\ref{eq:upperbound-S}) implies that $\beta\leq C \tilde{\beta}$, where
$\tilde{\beta} = \mathbb{E}(\tilde{\alpha})$ where $\tilde{\alpha}$ is defined by
\[
\tilde{\alpha} =   \frac{1}{N}\int_{0}^{1}\sum_{m=1}^{M}\left\vert
\sum_{i_{2},j_{2}}\frac{\partial\xi}{\partial\left(  \mathbf{W}_{i_{2},j_{2}
}^{m}\right)  ^{\ast}}\,e^{-2\pi\mathrm{i}(i_{2}+j_{2})\nu}\right\vert
^{2}\,d\nu
\]
and where $C$ is a nice constant. It is clear that $\tilde{\alpha}$ coincides
$\alpha$ when $\mathcal{S}_{m}(\nu)=1$ for each $m=1,\ldots,M$ and each $\nu
\in\lbrack0,1]$. When $\xi=\frac{1}{ML}\mathrm{Tr}(\mathbf{AQ}(z))$, it is
proved in \cite{loubaton16} that
\[
\tilde{\alpha} \leq\frac{1}{MN} \; \frac{1}{ML}\mathrm{tr}\left(
\mathbf{QAQW}\mathbf{W}^{H}\mathbf{Q}^{H}\mathbf{A}^{H}\mathbf{Q}^{H}\right).
\]
As it holds that $\mathbf{QWW}^{H}=\mathbf{I}+z\mathbf{Q}$ and that
$\Vert\mathbf{Q}\Vert\leq\frac{1}{\delta_{z}}$, we obtain that
\[
\mathbf{Q}\mathbf{W}\mathbf{W}^{H}\mathbf{Q}^{H}\leq\frac{1}{\delta_{z}
}\left(  1+\frac{|z|}{\delta_{z}}\right)  \mathbf{I}_{ML}\mathbf{.}
\]
Therefore,
\[
\tilde{\alpha} \leq\frac{1}{\delta_{z}}\left(  1+\frac{|z|}{\delta_{z}}\right)
\frac{1}{MN} \;  \frac{1}{ML}\mathrm{tr}(\mathbf{Q}
\mathbf{A}\mathbf{A}^{H}\mathbf{Q}^{H})
\]
and using again $\Vert\mathbf{Q}\Vert\leq\delta_{z}^{-1}$,
\[
\tilde{\beta} = \mathbb{E}(\tilde{\alpha})  \leq \frac{1}{\delta_{z}^{3}}\left(  1+\frac{|z|}{\delta_{z}
}\right)  \frac{1}{MN}\frac{1}{ML}\mathrm{tr}(\mathbf{A}\mathbf{A}^{H}).
\]
The conclusion follows from the observation that
\[
\frac{1}{\delta_{z}^{3}}\left(  1+\frac{|z|}{\delta_{z}}\right)  \leq\left[
\frac{1}{\delta_{z}^{3}}+\frac{1}{\delta_{z}^{4}}\right]  (|z|+1).
\]
As for the case $\xi=\frac{1}{ML}\mathrm{Tr}(\mathbf{AQ}(z)\mathbf{W}
\mathbf{G}\mathbf{W}^{H})$, we refer to upper bound of the term equivalent to $\tilde{\alpha}$ expressed
in Eq. (3.12-3.13) in \cite{loubaton16}, and omit further details.

\section{Expectation of resolvent and co-resolvent}

In this section, we analyze the expectation of the resolvent $\mathbf{Q}
(z)=\left(  \mathbf{WW}^{H}-z\mathbf{I}_{ML}\right)  ^{-1}$and co-resolvent
$\widetilde{\mathbf{Q}}(z)=\left(  \mathbf{W}^{H}\mathbf{W}-z\mathbf{I}
_{N}\right)  ^{-1}$. As a previous step, we need to ensure the properties of
certain useful matrix valued functions. This is summarized in the following lemma.

\begin{lem}
\label{lemma:invertibility_RRtilde}For $z\in\mathbb{C}\setminus\mathbb{R}^{+}
$, the matrix $\mathbf{I}_{N}+c_{N}\overline{\Psi}\left(  \mathbb{E}
\mathbf{Q}(z)\right)  $ is invertible, so that we can define
\begin{equation}
\widetilde{\mathbf{R}}(z)=\frac{-1}{z}\left(  \mathbf{I}_{N}+c_{N}
\overline{\Psi}\left(  \mathbb{E}\mathbf{Q}(z)\right)  \right)  ^{-1}
.\label{eq:def_Rztilde}
\end{equation}
On the other hand, the matrix $\mathbf{I}_{ML}+\Psi\left(  \widetilde
{\mathbf{R}}^{T}(z)\right)  $ is also invertible, and we define
\begin{equation}
\mathbf{R}(z)=\frac{-1}{z}\left(  \mathbf{I}_{ML}+\Psi\left(  \widetilde
{\mathbf{R}}^{T}(z)\right)  \right)  ^{-1}.\label{eq:def_Rz}
\end{equation}
Furthermore, $\widetilde{\mathbf{R}}(z)$ and $\mathbf{R}(z)$ are elements of
$\mathcal{S}_{N}(\mathbb{R}^{+})$ and $\mathcal{S}_{ML}(\mathbb{R}^{+})$
respectively. In particular, they are holomorphic on $\mathbb{C}
\setminus\mathbb{R}^{+}$ and satisfy
\begin{equation}
\mathbf{R}(z)\mathbf{R}^{H}(z)\leq\frac{\mathbf{I}_{ML}}{\delta_{z}^{2}
},\;\widetilde{\mathbf{R}}(z)\widetilde{\mathbf{R}}^{H}(z)\leq\frac
{\mathbf{I}_{N}}{\delta_{z}^{2}}\label{eq:upperbound-R-tildeR}
\end{equation}
Moreover, there exist two nice constants (see Definition \ref{def:nice}) $\eta$ and $\tilde{\eta}$ such that
\begin{align}
\mathbf{R}(z)\mathbf{R}^{H}(z) &  \geq\frac{\delta_{z}^{2}}{16(\eta
^{2}+|z|^{2})^{2}}\mathbf{I}_{ML}\label{eq:lower-bound-RR*}\\
\widetilde{\mathbf{R}}(z)\widetilde{\mathbf{R}}^{H}(z) &  \geq\frac{\delta
_{z}^{2}}{16(\tilde{\eta}^{2}+|z|^{2})^{2}}\mathbf{I}_{N}
.\label{eq:lower-bound-tildeRtildeR*}
\end{align}
\end{lem}

\begin{proof}
If $z \in \mathbb{R}^{-*}$, the invertibility of $\mathbf{I}_{N}+c_{N}\overline{\Psi}\left(
\mathbb{E}\mathbf{Q}(z)\right)  $ is obvious. If $z \in \mathbb{C}^{+}$, it follows from the fact that
\[
\mathrm{Im}\left[  \mathbf{I}_{N}+c_{N}\overline{\Psi}\left(  \mathbb{E}%
\mathbf{Q}(z)\right)  \right]  =c_{N}\overline{\Psi}\left(  \mathbb{E}%
\mathrm{Im}\mathbf{Q}(z)\right)
\]
and $\mathrm{Im}\mathbf{Q}(z)>0$. We now establish that $\widetilde{\mathbf{R}}(z)$ and $\mathbf{R}(z)$ are
elements of $\mathcal{S}_N(\mathbb{R}^{+})$ and $\mathcal{S}_{ML}(\mathbb{R}^{+})$. 
By Proposition \ref{prop:class-S}, we only need to prove that
$\mathrm{Im}\widetilde{\mathbf{R}}(z)\geq 0, \mathrm{Im}z\widetilde{\mathbf{R}%
}(z) \geq 0$ when $\mathrm{Im}z>0$, $\lim_{y\rightarrow+\infty}-\mathrm{i}%
y\widetilde{\mathbf{R}}(\mathrm{i}y)=\mathbf{I}_{N}$, and similar properties for matrix $\mathbf{R}(z)$. Clearly,
\begin{align*}
\mathrm{Im}\widetilde{\mathbf{R}}(z)  &  =\widetilde{\mathbf{R}}^{H}(z)\left[
\mathrm{Im}z\mathbf{I}_{N}+c_{N}\overline{\Psi}\left(  \mathrm{Im}\left[
z\mathbb{E}\mathbf{Q}(z)\right]  \right)  \right]  \widetilde{\mathbf{R}%
}(z)>0\\
\mathrm{Im}z\widetilde{\mathbf{R}}(z)  &  =c_{N}\left\vert z\right\vert
^{2}\widetilde{\mathbf{R}}^{H}(z)\left[  \overline{\Psi}\left(  \mathrm{Im}%
\mathbb{E}\mathbf{Q}(z)\right)  \right]  \widetilde{\mathbf{R}}(z)>0
\end{align*}
whereas, noting that $\mathbb{E}\mathbf{Q}(\mathrm{i}y)\rightarrow0$ as
$y\rightarrow+\infty$, we see that%
\[
-\mathrm{i}y\widetilde{\mathbf{R}}(\mathrm{i}y)=\left(  \mathbf{I}_{N}%
+c_{N}\overline{\Psi}\left(  \mathbb{E}\mathbf{Q}(\mathrm{i}y)\right)
\right)  ^{-1}\rightarrow\mathbf{I}_{N}%
\]
as $y\rightarrow+\infty$. In order to justify that $\mathbf{I}%
_{ML}+\Psi\left(  \widetilde{\mathbf{R}}^{T}(z)\right)  $ is invertible,
we remark that $\mathrm{Im}\left( \mathbf{I}%
_{ML}+\Psi\left(  \widetilde{\mathbf{R}}^{T}(z)\right) \right) $ 
coincides with $\Psi\left(  \mathrm{Im}(\widetilde{\mathbf{R}}^{T}(z)) \right)$
which is positive definite because $\mathrm{Im}\widetilde{\mathbf{R}}(z) > 0$ 
(see (\ref{eq:property_positive})). Therefore, 
 $\mathrm{Im}\left( \mathbf{I}%
_{ML}+\Psi\left(  \widetilde{\mathbf{R}}^{T}(z)\right) \right) > 0$
and $\mathbf{I}%
_{ML}+\Psi\left(  \widetilde{\mathbf{R}}^{T}(z)\right)  $ is invertible. 
Finally, observing that
\begin{align*}
\mathrm{Im}\mathbf{R}(z)  &  =\mathbf{R}^{H}(z)\left[  \mathrm{Im}%
z\mathbf{I}_{ML}+\Psi\left(  \mathrm{Im}\left[  z\widetilde{\mathbf{R}}%
^{T}(z)\right]  \right)  \right]  \mathbf{R}(z)>0\\
\mathrm{Im}z\mathbf{R}(z)  &  =\left\vert z\right\vert ^{2}\mathbf{R}%
^{H}(z)\left[  \Psi\left(  \mathrm{Im}\widetilde{\mathbf{R}}^{T}(z)\right)
\right]  \mathbf{R}(z)>0
\end{align*}
together with the fact that, since $\widetilde{\mathbf{R}}(\mathrm{i}%
y)\rightarrow0$ as $y\rightarrow\infty$,
\[
-\mathrm{i}y\mathbf{R}(\mathrm{i}y)=\left(  \mathbf{I}_{ML}+\Psi\left(
\widetilde{\mathbf{R}}^{T}(\mathrm{i}y)\right)  \right)  ^{-1}\rightarrow
\mathbf{I}_{ML}%
\]
We eventually establish (\ref{eq:lower-bound-RR*}), and omit the proof of 
(\ref{eq:lower-bound-tildeRtildeR*}). For this, we notice that ${\bf R}(z)$ is a 
block-diagonal matrix, and that measure ${\bs \nu}$ defined by 
${\bf R}(z) = \int_{\mathbb{R}^{+}} \frac{d {\bs \nu}(\lambda)}{\lambda - z}$ 
is block diagonal as well. In order to establish  (\ref{eq:lower-bound-RR*}), it is thus 
sufficient to prove that for each unit-norm $L$-dimensional vector ${\bf b}$, it holds that 
\begin{equation}
\label{eq:lower-bound-RmRm*}
{\bf b}^{H} {\bf R}^{m,m}(z) ({\bf R}^{m,m}(z))^{H} {\bf b}  \geq  \frac{\delta_z^{2}}{16(\eta^{2} + |z|^{2})^{2}}
\end{equation}
for some nice constant $\eta$ (of course independent on $m$ and ${\bf b}$). For this, 
we remark that 
$$
{\bf b}^{H} {\bf R}^{m,m}(z) ({\bf R}^{m,m}(z))^{H} {\bf b} \geq \left| {\bf b}^{H} {\bf R}^{m,m}(z) {\bf b} \right|^{2}
$$
We denote $\xi_m$ the term $\xi_m(z) = {\bf b}^{H} {\bf R}^{m,m}(z) {\bf b}$ which can be written as 
$$
\xi_m(z) = \int_{{\bf R}^{+}} \frac{d \mu_{\xi_m}(\lambda) }{\lambda - z}
$$
where probability measure $\mu_{\xi_m}$ is defined by $d \mu_{\xi_m}(\lambda) = {\bf b}^{H} d {\bs \nu}^{m,m}(\lambda) {\bf b}$.
We claim that 
\begin{equation}
\label{eq:useful-trick-lower-bound-xi}
|\xi_m(z)| \geq \delta_z \,  \int_{{\bf R}^{+}} \frac{{\bf b}^{H} d {\bs \nu}^{m,m}(\lambda) {\bf b}}{|\lambda - z|^{2}}
\end{equation} 
To justify this, we first remark that $\delta_z = |\mathrm{Im}(z)|$ if $\mathrm{Re}(z) \geq 0$ and that 
$\delta_z = |z|$ if $\mathrm{Re}(z) \leq 0$. 
Next, we notice that $|\xi_m(z)| \geq | \mathrm{Im}(\xi_m(z)) | = |\mathrm{Im}(z)| \int_{\mathbb{R}^{+}} \frac{d \mu_{\xi_m}(\lambda) }{|\lambda - z|^{2}}$ whatever the sign of $\mathrm{Re}(z)$. Therefore, if $\mathrm{Re}(z) \geq 0$, 
it holds that 
\[
|\xi_m(z)| \geq \delta_z \, \int_{\mathbb{R}^{+}} \frac{d \mu_{\xi_m}(\lambda) }{|\lambda - z|^{2}}
\]
If  $\mathrm{Re}(z) \leq 0$,  $\mathrm{Re}(\xi_m(z)) =  \int_{\mathbb{R}^{+}} \frac{\lambda - \mathrm{Re}(z) }{|\lambda - z|^{2}} d \mu_{\xi_m}(\lambda)$ verifies $\mathrm{Re}(\xi_m(z)) \geq -\mathrm{Re}(z) \, \int_{\mathbb{R}^{+}} \frac{d \mu_{\xi_m}(\lambda) }{|\lambda - z|^{2}}$. 
Therefore, if $\mathrm{Re}(z) \leq 0$, 
\[
|\xi_m(z)|^{2} = \left(\mathrm{Im}(\xi_m(z)\right)^{2} + \left(\mathrm{Re}(\xi_m(z)\right)^{2} \geq 
|z|^{2} \left( \int_{\mathbb{R}^{+}} \frac{d \mu_{\xi_m}(\lambda) }{|\lambda - z|^{2}} \right)^{2} = \delta_z^{2} \left( \int_{\mathbb{R}^{+}} \frac{d \mu_{\xi_m}(\lambda) }{|\lambda - z|^{2}} \right)^{2}
\]
Therefore, (\ref{eq:useful-trick-lower-bound-xi}) holds.  We now consider the family of probability measures \\
$\{ ({\bf b}_N^{H}   d {\bs \nu}_N^{m,m}(\lambda) {\bf b}_N)_{N \geq 1, m=1, \ldots, M, \| {\bf b}_N = 1 \|} \}$ where we have mentioned 
the dependency of ${\bf b}$ and ${\bs \nu}$ w.r.t. $N$. Using item (v) of Proposition 
\ref{prop:class-S} and hypothesis (\ref{eq:upperbound-S}) , it is easily seen that 
\[
\int_{\mathbb{R}^{+}} \lambda \; {\bf b}_N^{H}   d {\bs \nu}_N^{m,m}(\lambda) {\bf b}_N = {\bf b}_N^{H} {\bs \Psi}_L^{m}({\bf I}_N)  {\bf b}_N 
< C
\]
for some nice constant $C$. Therefore, it holds that 
\[
\sup_{N \geq 1, m=1, \ldots, M, \| {\bf b}_N \| = 1}  \int_{\mathbb{R}^{+}} \lambda \; {\bf b}_N^{H}   d {\bs \nu}_N^{m,m}(\lambda) {\bf b}_N < +\infty
\]
The family of probability measures is thus tight, and it exists a nice constant $\eta$ such that 
\begin{equation}
\label{eq:tightness-consequence}
\inf_{N \geq 1, m=1, \ldots, M, \| {\bf b}_N \| = 1} {\bf b}_N^{H} {\bs \nu}_N^{m,m}([0, \eta]) {\bf b}_N > 1/2
\end{equation}
We now use the obvious inequality 
\[
\int_{{\bf R}^{+}} \frac{{\bf b}^{H} d {\bs \nu}^{m,m}(\lambda) {\bf b}}{|\lambda - z|^{2}} \geq
\int_{0}^{\eta} \frac{{\bf b}^{H} d {\bs \nu}^{m,m}(\lambda) {\bf b}}{|\lambda - z|^{2}} 
\]
If $\lambda \in [0, \eta]$, it is clear that $|\lambda - z|^{2} \leq 2(|z|^{2} + \eta^{2})$, and 
that 
\[
\int_{0}^{\eta} \frac{{\bf b}^{H} d {\bs \nu}^{m,m}(\lambda) {\bf b}}{|\lambda - z|^{2}}  \geq \frac{1}{4(|z|^{2} + \eta^{2})}
\]
(\ref{eq:useful-trick-lower-bound-xi}) eventually leads to 
\[
{\bf b}^{H} {\bf R}^{m,m}(z) ({\bf R}^{m,m}(z))^{H} {\bf b} \geq |{\bf b}^{H} {\bf R}^{m,m}(z) {\bf b}|^{2} \geq \frac{\delta_z^{2}}{16(|z|^{2} + \eta^{2})^{2}}
\]
as expected.

\end{proof}

In order to address the expectation of $\mathbf{Q}(z)$ and $\widetilde
{\mathbf{Q}}(z)$ we will apply the integration by parts formula for the
expectations of Gaussian functions, which is presented next.

\begin{lem}
\label{lemma:ipp}Let $\xi=\xi\left(  \mathbf{W,W}^{\ast}\right)  $ denote a
$\mathcal{C}^{1}$ complex function such that both itself and its derivatives
are polynomically bounded. Under the above assumptions, we can write
\[
\mathbb{E}\left[  \mathbf{W}_{i_{1},j_{1}}^{m}\xi\right]  =\sum_{i_{2}=1}
^{L}\sum_{j_{2}=1}^{N}\mathbb{E}\left[  \mathbf{W}_{i_{1},j_{1}}^{m}\left(
\mathbf{W}_{i_{2},j_{2}}^{m}\right)  ^{\ast}\right]  \mathbb{E}\left[
\frac{\partial\xi}{\partial\left(  \mathbf{W}_{i_{2},j_{2}}^{m}\right)
^{\ast}}\right]
\]
where $\mathbf{W}_{i,j}^{m}$ is the $((m-1)L+i,j)$th entry of $\mathbf{W}$.
\end{lem}

\begin{proof}
See \cite{pastur-shcherbina-book, hachem08}.
\end{proof}

Consider the resolvent identity
\begin{equation}
z\mathbf{Q}(z)=\mathbf{Q}(z)\mathbf{WW}^{H}-\mathbf{I}_{ML}
.\label{eq:resolvent_identity}
\end{equation}
Let $\mathbf{w}_{k}$ denote the $k$th column of matrix $\mathbf{W}$, $1\leq
k\leq N$. For an $ML\times ML$ matrix $\mathbf{A}$, we recall that we denote as $\left[
\mathbf{A}\right]  ^{m_{1},m_{2}}$ its $\left(  m_{1},m_{2}\right)  $th block
matrix (of size $L\times L$) and as $\left[  \mathbf{A}\right]  _{i_{1},i_{2}
}^{m_{1},m_{2}}$ the $\left(  i_{1},i_{2}\right)  $th entry of its $\left(
m_{1},m_{2}\right)  $th block. Applying the integration by parts formula in
Lemma \ref{lemma:ipp} and the identity in (\ref{eq:correlationWW}), we are
able to write
\begin{align*}
\mathbb{E}\left[  \mathbf{Q}(z)\mathbf{w}_{k}\mathbf{w}_{j}^{H}\right]
_{i_{1},i_{2}}^{m_{1},m_{2}} &  =\sum_{m_{3},i_{3}}\mathbb{E}\left[  \left[
\mathbf{Q}(z)\right]  _{i_{1},i_{3}}^{m_{1},m_{3}}\mathbf{W}_{i_{3},k}^{m_{3}
}\left(  \mathbf{W}_{i_{2},j}^{m_{2}}\right)  ^{\ast}\right]  \\
&  =\sum_{r=1}^{N}\sum_{i_{4}=1}^{L}\sum_{m_{3},i_{3}}\mathbb{E}\left[
\mathbf{W}_{i_{3},k}^{m_{3}}\left(  \mathbf{W}_{i_{4},r}^{m_{3}}\right)
^{\ast}\right]  \mathbb{E}\left[  \frac{\partial\left[  \mathbf{Q}(z)\right]
_{i_{1},i_{3}}^{m_{1},m_{3}}\left(  \mathbf{W}_{i_{2},j}^{m_{2}}\right)
^{\ast}}{\partial\left(  \mathbf{W}_{i_{4},r}^{m_{3}}\right)  ^{\ast}}\right]
\\
&  =-\sum_{r=1}^{N}\sum_{m_{3}=1}^{M}\sum_{i_{3}=1}^{L}\sum_{i_{4}=1}^{L}
\frac{r_{m_{3}}\left(  k-r+i_{3}-i_{4}\right)  }{N}\mathbb{E}\left[  \left[
\mathbf{Q}(z)\mathbf{w}_{r}\mathbf{w}_{j}^{H}\right]  _{i_{1},i_{2}}
^{m_{1,}m_{2}}\left[  \mathbf{Q}(z)\right]  _{i_{4},i_{3}}^{m_{3},m_{3}
}\right]  \\
&  +\sum_{i_{3}=1}^{L}\frac{r_{m_{2}}\left(  k-j+i_{3}-i_{2}\right)  }
{N}\mathbb{E}\left[  \mathbf{Q}(z)\right]  _{i_{1},i_{3}}^{m_{1},m_{2}}.
\end{align*}
Now, using the change of variable $i=i_{4}-i_{3}$ we can alternatively express
\begin{align*}
&  \mathbb{E}\left[  \mathbf{Q}(z)\mathbf{w}_{k}\mathbf{w}_{j}^{H}\right]
_{i_{1},i_{2}}^{m_{1},m_{2}}\\
&  =-L\sum_{r=1}^{N}\sum_{m_{3}=1}^{M}\sum_{i=-L+1}^{L-1}\frac{r_{m_{3}
}\left(  k-r-i\right)  }{N}\mathbb{E}\left[  \left[  \mathbf{Q}(z)\mathbf{w}
_{r}\mathbf{w}_{j}^{H}\right]  _{i_{1},i_{2}}^{m_{1,}m_{2}}\tau\left(
\mathbf{Q}^{m_{3},m_{3}}(z)\right)  \left(  i\right)  \right]  \\
&  +\sum_{i_{3}=1}^{L}\frac{r_{m_{2}}\left(  k-j+i_{3}-i_{2}\right)  }
{N}\mathbb{E}\left[  \mathbf{Q}(z)\right]  _{i_{1},i_{3}}^{m_{1},m_{2}}
\end{align*}
where we recall that, for a given square matrix $\mathbf{X}$ of size $R$, the
sequence $\tau\left(  \mathbf{X}\right)  \left(  i\right)  $ is defined in
(\ref{eq:definition_tau}).

Using the definition of the operator $\Psi_{N}^{(m)}$ and its averaged
counterpart in (\ref{eq:def_Phi_average}), we may reexpress the above equation
as
\begin{align}
\mathbb{E}\left[  \mathbf{Q}(z)\mathbf{w}_{k}\mathbf{w}_{j}^{H}\right]
_{i_{1},i_{2}}^{m_{1},m_{2}} &  =-c_{N}\mathbb{E}\left[  \overline{\Psi
}\left(  \mathbf{Q}(z)\right)  \mathbf{W}^{T}\mathbf{Q}^{T}(z)\mathbf{e}
_{i_{1}}^{m_{1}}\left(  \mathbf{e}_{i_{2}}^{m_{2}}\right)  ^{T}\mathbf{W}
^{\ast}\right]  _{k,j}\label{eq:expectation_resolvent}\\
&  +\sum_{i_{3}=1}^{L}\frac{r_{m_{2}}\left(  k-j+i_{3}-i_{2}\right)  }
{N}\mathbb{E}\left[  \left(  \mathbf{Q}(z)\right)  _{i_{1},i_{3}}^{m_{1}
,m_{2}}\right]  \nonumber
\end{align}
where we recall that $c_{N}=\frac{ML}{N}$. From
(\ref{eq:expectation_resolvent}) and using the definition of $\Psi
($\textperiodcentered$)$, we may generally write, for any $N\times N$
deterministic matrix $\mathbf{A}$
\begin{equation}
\mathbb{E}\left[  \mathbf{Q}(z)\mathbf{WAW}^{H}\right]  =-c_{N}\mathbb{E}
\left[  \mathbf{Q}(z)\mathbf{W}\overline{\Psi}^{T}\left(  \mathbf{Q}
(z)\right)  \mathbf{AW}^{H}\right]  +\mathbb{E}\left[  \mathbf{Q}
(z)\Psi\left(  \mathbf{A}^{T}\right)  \right]  .\label{eq:resolvent_general_A}
\end{equation}

Let us now consider the co-resolvent, namely $\widetilde{\mathbf{Q}
}(z)=\left(  \mathbf{W}^{H}\mathbf{W}-z\mathbf{I}_{N}\right)  ^{-1}$, together
with the co-resolvent identity
\[
z\widetilde{\mathbf{Q}}(z)=\widetilde{\mathbf{Q}}(z)\mathbf{W}^{H}
\mathbf{W}-\mathbf{I}_{N}.
\]
Observe that we can write $\widetilde{\mathbf{Q}}(z)\mathbf{W}^{H}
\mathbf{W}=\mathbf{W}^{H}\mathbf{Q}(z)\mathbf{W}$ and therefore
\[
\mathbb{E}\left[  \widetilde{\mathbf{Q}}(z)\mathbf{W}^{H}\mathbf{W}\right]
_{j,k}=\mathbb{E}\left[  \mathbf{W}^{H}\mathbf{Q}(z)\mathbf{W}\right]
_{j,k}=\mathbb{E}\mathrm{tr}\left[  \mathbf{Q}(z)\mathbf{w}_{k}\mathbf{w}
_{j}^{H}\right]  .
\]
Hence, using the expression for the expectation of the resolvent in
(\ref{eq:expectation_resolvent}), we can obtain
\begin{align*}
\mathbb{E}\left[  \widetilde{\mathbf{Q}}(z)\mathbf{W}^{H}\mathbf{W}\right]
_{j,k} &  =\mathbb{E}\mathrm{tr}\left[  \mathbf{Q}(z)\mathbf{w}_{k}
\mathbf{w}_{j}^{H}\right]  =-c_{N}\mathbb{E}\left[  \widetilde{\mathbf{Q}
}(z)\mathbf{W}^{H}\mathbf{W}\overline{\Psi}^{T}\left(  \mathbf{Q}(z)\right)
\right]  _{j,k}\\
&  +\sum_{m_{1}=1}^{M}\sum_{i_{1}=1}^{L}\sum_{i_{3}=1}^{L}\frac{r_{m_{1}
}\left(  k-j+i_{3}-i_{1}\right)  }{N}\mathbb{E}\left[  \left(  \mathbf{Q}
(z)\right)  _{i_{1},i_{3}}^{m_{1},m_{1}}\right]
\end{align*}
The second term can be further simplified by applying the change of variables
$i_{1}=i+i_{3}$, namely
\[
\mathbb{E}\left[  \widetilde{\mathbf{Q}}(z)\mathbf{W}^{H}\mathbf{W}\right]
=-c_{N}\mathbb{E}\left[  \widetilde{\mathbf{Q}}(z)\mathbf{W}^{H}
\mathbf{W}\overline{\Psi}^{T}\left(  \mathbf{Q}(z)\right)  \right]
+c_{N}\overline{\Psi}^{T}\left(  \mathbb{E}\mathbf{Q}(z)\right)
\]
and therefore, by the resolvent's identity,
\[
\mathbb{E}\left[  \widetilde{\mathbf{Q}}(z)\right]  =-\frac{1}{z}
\mathbf{I}_{N}-c_{N}\mathbb{E}\left[  \widetilde{\mathbf{Q}}(z)\overline{\Psi
}^{T}\left(  \mathbf{Q}(z)\right)  \right]  .
\]
Now, replacing $\mathbf{Q}(z)$ in the above equation by $\mathbf{Q}
(z)=\mathbb{E}\mathbf{Q}(z)+\mathbf{Q}^{\circ}(z)$ (where $X^{\circ
}=X-\mathbb{E}X$) we see that
\[
\mathbb{E}\left[  \widetilde{\mathbf{Q}}(z)\right]  =\widetilde{\mathbf{R}
}(z)+zc_{N}\mathbb{E}\left[  \widetilde{\mathbf{Q}}(z)\overline{\Psi}
^{T}\left(  \mathbf{Q}^{\circ}(z)\right)  \widetilde{\mathbf{R}}(z)\right]
\]
where $\widetilde{\mathbf{R}}(z)$ is defined in (\ref{eq:def_Rztilde}). On the
other hand, particularizing the equation in (\ref{eq:resolvent_general_A}) to
the case $\mathbf{A}=\widetilde{\mathbf{R}}(z)$ and using the resolvent's
identity in (\ref{eq:resolvent_identity}), we also obtain
\[
\mathbb{E}\left[  \mathbf{Q}(z)\right]  =\mathbf{R}(z)-zc_{N}\mathbb{E}\left[
\mathbf{Q}(z)\mathbf{W}\overline{\Psi}^{T}\left(  \mathbf{Q}^{\circ
}(z)\right)  \widetilde{\mathbf{R}}(z)\mathbf{W}^{H}\mathbf{R}(z)\right]
\]
where $\mathbf{R}(z)$ is defined in (\ref{eq:def_Rz}). With this, we have
arrived at the two fundamental equations, which are summarized in what
follows:
\[
\mathbb{E}\left[  \mathbf{Q}(z)\right]  -\mathbf{R}(z)=\boldsymbol\Delta
(z),\quad\mathbb{E}\left[  \widetilde{\mathbf{Q}}(z)\right]  -\widetilde
{\mathbf{R}}(z)=\widetilde{\boldsymbol\Delta}(z)
\]
where the error terms are defined as
\begin{align}
\boldsymbol\Delta(z) &  =-zc_{N}\mathbb{E}\left[  \mathbf{Q}(z)\mathbf{W}
\overline{\Psi}^{T}\left(  \mathbf{Q}^{\circ}(z)\right)  \widetilde
{\mathbf{R}}(z)\mathbf{W}^{H}\mathbf{R}(z)\right]  \label{eq:expre-Delta}\\
\widetilde{\boldsymbol\Delta}(z) &  =zc_{N}\mathbb{E}\left[  \widetilde
{\mathbf{Q}}(z)\overline{\Psi}^{T}\left(  \mathbf{Q}^{\circ}(z)\right)
\widetilde{\mathbf{R}}(z)\right]  .
\end{align}

\subsection{Control of the errors}

We develop here a control on certain functionals of the error term $\Delta
(z)$. In this section, we establish the following result.

\begin{prop}
\label{prop:control-Delta} For each deterministic sequence of $ML\times ML$
matrices $(\mathbf{A}_{N})_{N\geq1}$ satisfying $\sup_{N}\Vert\mathbf{A}
_{N}\Vert<a<+\infty$, it holds that
\begin{equation}
\left\vert \frac{1}{ML}\mathrm{tr}\left[  \mathbf{A}_N\boldsymbol\Delta
(z)\right]  \right\vert \leq a\,C(z)\frac{L}{MN}\label{eq:control-trace-Delta}
\end{equation}
where $C(z)=P_{1}(|z|)P_{2}(1/\delta_{z})$ for some nice polynomials $P_{1}$
and $P_{2}$ (see Definition \ref{def:nice}) that do not depend on sequence $(\mathbf{A}_{N})_{N\geq1}$.
Moreover, if $(\mathbf{b}_{1,N})_{N \geq 1}$ and $(\mathbf{b}
_{2,N})_{N \geq 1}$ are 2 sequences of $L$ dimensional vectors such that
$\sup_{N}\Vert\mathbf{b}_{i,N}\Vert<b<+\infty$ for $i=1,2$, and if
$((d_{m,N})_{m=1,\ldots, M})_{N\geq1}$ are deterministic complex number verifying
$\sup_{N,m}|d_{m,N}|<d<+\infty$, then, it holds that
\begin{equation}
\left\vert \mathbf{b}_{1,N}^{H}\left(  \frac{1}{M}\sum_{m=1}^{M}d_{m,N}{\Delta
}^{m,m}(z)\right)  \mathbf{b}_{2,N}\right\vert \leq d\,b^{2}\,C(z)\frac{L^{3/2}
}{MN}\label{eq:controle-trace-partielle-Delta}
\end{equation}
where $C(z)$ is defined as above, where the nice polynomials $P_{1}$ and
$P_{2}$ do not depend on $(\mathbf{b}_{1,N})_{N \geq 1}$, $(\mathbf{b}
_{2,N})_{N \geq 1}$ and $((d_{m,N})_{m=1,\ldots, M})_{N\geq1}$
\end{prop}

We first establish (\ref{eq:control-trace-Delta}), and denote ${\bf A}_N, ({\bf b}_{i,N})_{i=1,2}$ and $(d_{m,N})_{m=1, \ldots, M}$
by  ${\bf A}, ({\bf b}_{i})_{i=1,2}$ and $(d_{m})_{m=1, \ldots, M}$ in order to simplify the notations. We denote by $\xi$ the term
$\frac{1}{ML}\mathrm{tr}\left[  \mathbf{A}\Delta(z)\right]  $, and express
$\xi$ as
\[
\xi=\frac{1}{ML}\mathrm{tr}\left(  \overline{\Psi}^{T}(\mathbf{Q}^{\circ
})\widetilde{\mathbf{R}}\mathbf{W}^{H}\mathbf{R}\mathbf{A}\mathbf{Q}
\mathbf{W}\right)  .
\]
Using (\ref{eq:expre-1-Psibar}), we obtain immediately that
\[
\xi=\sum_{n=-(N-1)}^{N-1}\sum_{l=-(L-1)}^{L-1}\tau^{(M)}\left(  \mathbf{Q}
^{\circ}(\mathcal{R}(n-l)\otimes\mathbf{I}_{L})\right)  (l)\;\tau\left(
\widetilde{\mathbf{R}}\mathbf{W}^{H}\mathbf{R}\mathbf{A}\mathbf{Q}
\mathbf{W}\right)  (n).
\]
Therefore, $\mathbb{E}(\xi)$ is equal to
\[
\mathbb{E}(\xi)=\sum_{n=-(N-1)}^{N-1}\sum_{l=-(L-1)}^{L-1}\mathbb{E}\left[
\tau^{(M)}\left(  \mathbf{Q}^{\circ}(\mathcal{R}(n-l)\otimes\mathbf{I}
_{L})\right)  (l)\;\tau\left(  \widetilde{\mathbf{R}}\mathbf{W}^{H}
\mathbf{R}\mathbf{A}\mathbf{Q}\mathbf{W}\right)  ^{\circ}(n)\right]  .
\]
Using the definition of operators $\tau$ and $\tau^{(M)}$, the Schwartz
inequality, Lemma \ref{le:variance-trace}, and the inequality $\mathbf{J}
_{L}^{l}\mathbf{J}_{L}^{l(H)}\leq\mathbf{I}_{L}$, we obtain that
\begin{multline*}
\left\vert \mathbb{E}\left[  \tau^{(M)}\left(  \mathbf{Q}^{\circ}
(\mathcal{R}(n-l)\otimes\mathbf{I}_{L})\right)  (l)\;\tau\left(
\widetilde{\mathbf{R}}\mathbf{W}^{H}\mathbf{R}\mathbf{A}\mathbf{Q}
\mathbf{W}\right)  ^{\circ}(n)\right]  \right\vert \\
\leq C(z)\frac{1}{MN}\,\left(  \frac{1}{ML}\mathrm{Tr}(\mathcal{R}
(n-l)\mathcal{R}^{H}(n-l)\otimes\mathbf{I}_{L})\right)  ^{1/2}.
\end{multline*}
Therefore, it holds that
\[
|\mathbb{E}(\xi)|\leq\frac{C(z)}{MN}\sum_{n=-(N-1)}^{N-1}\sum_{l=-(L-1)}
^{L-1}\left(  \frac{1}{M}\sum_{m=1}^{M}|r_{m}(n-l)|^{2}\right)  ^{1/2}
\]
and that
\[
|\mathbb{E}(\xi)|\leq C(z)\frac{L}{MN}\sum_{n\in\mathbb{Z}}\left(  \frac{1}
{M}\sum_{m=1}^{M}|r_{m}(n)|^{2}\right)  ^{1/2}.
\]
Condition ({\ref{eq:condition-rm}}) thus implies (\ref{eq:control-trace-Delta}).

In order to establish (\ref{eq:controle-trace-partielle-Delta}), we denote by
$\eta$ the left hand side of (\ref{eq:controle-trace-partielle-Delta}), which
can be written as
\[
\eta=\frac{1}{M}\mathrm{tr}\left(  {\boldsymbol\Delta}(\mathbf{D}
\otimes\mathbf{b}_{2}\mathbf{b}_{1}^{H})\right)
\]
where $\mathbf{D}$ represents the $M\times M$ diagonal matrix with diagonal
entries $d_{1},\ldots,d_{M}$. Using the above calculations, we obtain that
$\mathbb{E}(\eta)$ can be expressed as
\[
\mathbb{E}(\eta)=L\;\sum_{n=-(N-1)}^{N-1}\sum_{l=-(L-1)}^{L-1}\mathbb{E}
\left[  \tau^{(M)}\left(  \mathbf{Q}^{\circ}(\mathcal{R}(n-l)\otimes
\mathbf{I}_{L})\right)  (l)\;\tau\left(  \widetilde{\mathbf{R}}\mathbf{W}
^{H}\mathbf{R}(\mathbf{D}\otimes\mathbf{b}_{2}\mathbf{b}_{1}^{H}
)\mathbf{Q}\mathbf{W}\right)  ^{\circ}(n)\right]  .
\]
We use again the Schwartz inequality to evaluate $\mathbb{E}\left[  \tau
^{(M)}\left(  \mathbf{Q}^{\circ}(\mathcal{R}(n-l)\otimes\mathbf{I}
_{L})\right)  (l)\;\tau\left(  \widetilde{\mathbf{R}}\mathbf{W}^{H}
\mathbf{R}(\mathbf{D}\otimes\mathbf{b}_{2}\mathbf{b}_{1}^{H})\mathbf{Q}
\mathbf{W}\right)  ^{\circ}(n)\right]  $ together with
(\ref{eq:var_trace_complicated_term}) for $\mathbf{A}=\mathbf{R}
(\mathbf{D}\otimes\mathbf{b}_{2}\mathbf{b}_{1}^{H})$, and obtain that
\[
\mathbb{E}\left\vert \tau\left(  \widetilde{\mathbf{R}}\mathbf{W}
^{H}\mathbf{R}(\mathbf{D}\otimes\mathbf{b}_{2}\mathbf{b}_{1}^{H}
)\mathbf{Q}\mathbf{W}\right)  ^{\circ}(n)\right\vert ^{2}\leq b^{2}\frac
{C(z)}{MN}\frac{1}{ML}\mathrm{tr}(\mathbf{D}\mathbf{D}^{H}\otimes
\mathbf{b}_{2}\mathbf{b}_{2}^{H})
\]
The term $\frac{1}{ML}\mathrm{tr}(\mathbf{D}\mathbf{D}^{H}\otimes
\mathbf{b}_{2}\mathbf{b}_{2}^{H})$ can also be written as
\[
\frac{1}{ML}\mathrm{tr}(\mathbf{D}\mathbf{D}^{H}\otimes\mathbf{b}
_{2}\mathbf{b}_{2}^{H})=\frac{1}{L}\left(  \frac{1}{M}\sum_{m=1}^{M}
|d_{m}|^{2}\right)  \Vert\mathbf{b}_{2}\Vert^{2}
\]
and is thus bounded by $\frac{b^{2}d^{2}}{L}$. Therefore, the Schwartz
inequality leads to
\begin{multline}
\left\vert \mathbb{E}\left[  \tau^{(M)}\left(  \mathbf{Q}^{\circ}
(\mathcal{R}(n-l)\otimes\mathbf{I}_{L})\right)  (l)\;\tau\left(
\widetilde{\mathbf{R}}\mathbf{W}^{H}\mathbf{R}(\mathbf{D}\otimes\mathbf{b}
_{2}\mathbf{b}_{1}^{H})\mathbf{Q}\mathbf{W}\right)  ^{\circ}(n)\right]
\right\vert  \leq \\ b^{2}d\,C(z)\frac{1}{MN\sqrt{L}}\left(  \frac{1}{M}\sum
_{m=1}^{M}|r_{m}(n-l)|^{2}\right)  ^{1/2}.
\end{multline}
Using the same approach as above, we immediately obtain
(\ref{eq:controle-trace-partielle-Delta}).

\begin{cor}
\label{cor:norm-psibar-delta} It holds that
\begin{equation}
\label{eq:norm-psibar-delta}\| \overline{\Psi}(\Delta) \| \leq C(z)
\frac{L^{3/2}}{MN}
\end{equation}

\end{cor}

(\ref{eq:norm-psibar-delta}) follows immediately from
\[
\Vert\overline{\Psi}(\Delta)\Vert\leq\sup_{\nu\in\lbrack0,1]}\frac{1}{M}
\sum_{m=1}^{M}\mathcal{S}_{m}(\nu)|\mathbf{a}_{L}^{H}(\nu){\boldsymbol\Delta
}^{m,m}\mathbf{a}_{L}(\nu)|
\]
and from the application of (\ref{eq:controle-trace-partielle-Delta}) to the
case $\mathbf{b}_{1}=\mathbf{b}_{2}=\mathbf{a}_{L}(\nu)$ and $d_{m}
=\mathcal{S}_{m}(\nu)$.

\section{The deterministic equivalents.}

As $\mathbb{E}({\bf Q}(z)) - {\bf R}(z)$ converges towards zero in some appropriate sense, (\ref{eq:def_Rztilde}) and 
(\ref{eq:def_Rz}) suggest that it is reasonable
to expect that $\mathbb{E}({\bf Q}(z))$ behaves as the first component ${\bf T}(z)$ of the solution $({\bf T}(z), \widetilde
{\mathbf{T}}(z)$) of the so-called 
canonical equation
\begin{align}
\mathbf{T}(z)  &  =-\frac{1}{z}\left(  \mathbf{I}_{ML}+\Psi\left(
\widetilde{\mathbf{T}}^{T}(z)\right)  \right)  ^{-1}\label{eq:canonical-T}\\
\widetilde{\mathbf{T}}(z)  &  =-\frac{1}{z}\left(  \mathbf{I}_{N}
+c_{N}\overline{\Psi}^{T}\left(  \mathbf{T}(z)\right)  \right)  ^{-1}.
\label{eq:canonical-tildeT}
\end{align}
In the following, we establish that the canonical equation has a unique solution. More precisely: 

\begin{prop}
\label{proposition:existence_unicity} There exists a unique pair of functions
$(\mathbf{T}(z),\widetilde{\mathbf{T}}(z))\in\mathcal{S}_{ML}(\mathbb{R}
^{+})\times\mathcal{S}_{N}(\mathbb{R}^{+})$ that satisfy (\ref{eq:canonical-T}
, \ref{eq:canonical-tildeT}) for each $z\in\mathbb{C}\setminus\mathbb{R}^{+}$.
Moreover, there exist two nice constants (see Definition \ref{def:nice}) $\eta$ and $\tilde{\eta}$ such that
\begin{align}
\mathbf{T}(z)\mathbf{T}^{H}(z) &  \geq\frac{\delta_{z}^{2}}{16(\eta
^{2}+|z|^{2})^{2}}\mathbf{I}\label{eq:lower-bound-TT*}\\
\widetilde{\mathbf{T}}(z)\widetilde{\mathbf{T}}^{H}(z) &  \geq\frac{\delta
_{z}^{2}}{16(\tilde{\eta}^{2}+|z|^{2})^{2}}\mathbf{I}\label{eq:lower-bound-tildeTtildeT*}
\end{align}

\end{prop}

We devote the rest of this section to proving this proposition. We will first
prove existence of a solution by using a standard convergence argument.

\begin{prop}
\label{prop:iteration}Let $\mathbf{\Gamma}^{m}(z),$ $m=1,\ldots,M$, be a
collection of $L\times L$ matrix-valued complex function belonging to
$\mathcal{S}_{L}\left(  \mathbb{R}^{\mathbb{+}}\right)  $ and define
$\mathbf{\Gamma}(z)=\mathrm{diag}\left(  \mathbf{\Gamma}^{1}(z),\ldots
,\mathbf{\Gamma}^{M}(z)\right)  $. Likewise, let $\widetilde{\mathbf{\Gamma}
}(z)$ be an $N\times N$ matrix-valued complex function belonging to
$\mathcal{S}_{N}\left(  \mathbb{R}^{\mathbb{+}}\right)  $. Consider the two
matrices
\begin{align*}
\mathbf{\Upsilon}(z) &  =-\frac{1}{z}\left(  \mathbf{I}_{ML}+\Psi\left(
\widetilde{\mathbf{\Gamma}}^{T}(z)\right)  \right)  ^{-1}\\
\widetilde{\mathbf{\Upsilon}}(z) &  =-\frac{1}{z}\left(  \mathbf{I}_{N}
+c_{N}\overline{\Psi}^{T}\left(  \mathbf{\Gamma}(z)\right)  \right)  ^{-1}
\end{align*}
and let $\mathbf{\Upsilon}(z)=\mathrm{diag}\left(  \mathbf{\Upsilon}
^{1}(z),\ldots,\mathbf{\Upsilon}^{M}(z)\right)  $. The matrix-valued functions
$\mathbf{\Upsilon}^{m}(z),$ $m=1,\ldots,M$ and $\widetilde{\mathbf{\Upsilon}
}(z)$ are analytic on $\mathbb{C}\backslash\mathbb{R}^{\mathbb{+}}$ and belong
to the classes $\mathcal{S}_{L}\left(  \mathbb{R}^{\mathbb{+}}\right)  $ and
$\mathcal{S}_{N}\left(  \mathbb{R}^{\mathbb{+}}\right)  $ respectively.
\end{prop}

\begin{proof}
The proof follows the lines of \cite[Proposition 5.1]
{hachem-loubaton-najim-aap-2007}. We first need to prove that
$\mathbf{\Upsilon}^{m}(z)$ and $\widetilde{\mathbf{\Upsilon}}(z)$ are analytic
on $\mathbb{C}\backslash\mathbb{R}^{\mathbb{+}}$. To see this, observe that if
$\mathbf{A}(z)$ is an analytic matrix-valued function, so is $\Psi_{K}
^{(m)}\left(  \mathbf{A}(z)\right)  ,m=1,\ldots,M.$ Therefore, we only need to
show $\mathrm{\det}\left(  z\mathbf{I}_{L}+\Psi_{L}^{(m)}\left(
z\widetilde{\mathbf{\Gamma}}^{T}(z)\right)  \right)  \neq0$ and $\mathrm{\det
}\left(  z\mathbf{I}_{N}+c_{N}\overline{\Psi}^{T}\left(  z\mathbf{\Gamma
}(z)\right)  \right)  \neq0$ when $z\in\mathbb{C}\backslash\mathbb{R}
^{\mathbb{+}}$. Let $\mathbf{h}$ denote an arbitrary $L\times1$ column vector
such that $\left(  z\mathbf{I}_{L}+\Psi_{L}^{(m)}\left(  z\widetilde
{\mathbf{\Gamma}}^{T}(z)\right)  \right)  \mathbf{h}=0$. If $z\in
\mathbb{C}^{\mathbb{+}}$, have
\begin{align*}
0 &  =\left\vert \mathbf{h}^{H}\left(  z\mathbf{I}_{L}+\Psi_{L}^{(m)}\left(
z\widetilde{\mathbf{\Gamma}}^{T}(z)\right)  \right)  \mathbf{h}\right\vert
\geq\mathrm{Im}\mathbf{h}^{H}\left(  z\mathbf{I}_{L}+\Psi_{L}^{(m)}\left(
z\widetilde{\mathbf{\Gamma}}^{T}(z)\right)  \right)  \mathbf{h}\\
&  \mathbf{=}\mathbf{h}^{H}\left(  \mathrm{Im}z\mathbf{I}_{L}+\Psi_{L}
^{(m)}\left(  \mathrm{Im}z\widetilde{\mathbf{\Gamma}}^{T}(z)\right)  \right)
\mathbf{h}\geq\mathrm{Im}z\left\Vert \mathbf{h}\right\Vert ^{2}\geq0
\end{align*}
where we have used $\mathrm{Im}z\widetilde{\mathbf{\Gamma}}^{T}(z)\geq
\mathbf{0}$ because $\widetilde{\mathbf{\Gamma}}(z)\in\mathcal{S}_{N}\left(
\mathbb{R}^{\mathbb{+}}\right)  $. From the above chain of inequalities we see
that we can only have $\mathbf{h=0}$. The same argument is valid when
$z\in\mathbb{C}^{\mathbb{-}}$. On the other hand, when $z\in\mathbb{R}
^{\mathbb{-}}$, we will have $-z\widetilde{\mathbf{\Gamma}}^{T}(z)\geq0$ and
\[
0=\mathbf{h}^{H}\left(  -z\mathbf{I}_{L}+\Psi_{L}^{(m)}\left(  -z\widetilde
{\mathbf{\Gamma}}^{T}(z)\right)  \right)  \mathbf{h}\geq\left\vert
z\right\vert \left\Vert \mathbf{h}\right\Vert ^{2}\geq0
\]
which also implies that $\mathbf{h=0}$. A similar argument proves that
$\mathrm{\det}\left(  z\mathbf{I}_{N}+c_{N}\overline{\Psi}^{T}\left(
z\mathbf{\Gamma}(z)\right)  \right)  \neq0$ when $z\in\mathbb{C}
\backslash\mathbb{R}^{\mathbb{+}}$.

Next, we prove that $\mathrm{Im}\mathbf{\Upsilon}^{m}(z)\geq0$ and
$\mathrm{Im}z\mathbf{\Upsilon}^{m}(z)\geq0$ when $z\in\mathbb{C}^{\mathbb{+}}
$. Observe that, using the identity $\mathbf{A}^{-1}-\mathbf{B}^{-1}
=\mathbf{A}^{-1}(\mathbf{B}-\mathbf{A})\mathbf{B}^{-1}$, we have
\[
\mathrm{Im}\mathbf{\Upsilon}^{m}(z)=\mathbf{\Upsilon}^{m}(z)^{H}\left[
\mathrm{Im}z\mathbf{I}_{L}+\Psi_{L}^{(m)}\left(  \mathrm{Im}\left[
z\widetilde{\mathbf{\Gamma}}^{T}(z)\right]  \right)  \right]  \mathbf{\Upsilon
}^{m}(z)\geq0
\]
because $\mathrm{Im}\left[  z\widetilde{\mathbf{\Gamma}}^{T}(z)\right]  \geq0$
since $\widetilde{\mathbf{\Gamma}}(z)\in\mathcal{S}_{N}\left(  \mathbb{R}
^{\mathbb{+}}\right)  $. On the the other hand, we also have
\[
\mathrm{Im}\left[  z\mathbf{\Upsilon}^{m}(z)\right]  =\left[
z\mathbf{\Upsilon}^{m}(z)\right]  ^{H}\left[  \Psi_{L}^{(m)}\left(
\mathrm{Im}\widetilde{\mathbf{\Gamma}}^{T}(z)\right)  \right]  \left[
z\mathbf{\Upsilon}^{m}(z)\right]  \geq0
\]
because $\mathrm{Im}\widetilde{\mathbf{\Gamma}}^{T}(z)\geq0$ since
$\widetilde{\mathbf{\Gamma}}(z)\in\mathcal{S}_{N}\left(  \mathbb{R}
^{\mathbb{+}}\right)  $.

Finally, one can readily see that $\lim_{y\rightarrow0}\mathrm{i}
y\mathbf{\Upsilon}^{m}(\mathrm{i}y)=-\mathbf{I}_{L}$ because $\lim
_{y\rightarrow0}\widetilde{\mathbf{\Gamma}}(\mathrm{i}y)=\mathbf{0}$.
Consequently, Proposition \ref{prop:class-S} implies that $\mathbf{\Upsilon
}^{m}(z),$ $m=1,\ldots,M$ belong to the class $\mathcal{S}_{L}\left(
\mathbb{R}^{\mathbb{+}}\right)  $. We can similarly prove that $\widetilde
{\mathbf{\Upsilon}}(z)$ belongs to $\mathcal{S}_{N}\left(  \mathbb{R}
^{\mathbb{+}}\right)  $.
\end{proof}

Let us now define the sequence of functions in $\mathcal{S}_{ML}\left(
\mathbb{R}^{\mathbb{+}}\right)  $ that will lead to a solution. We begin by
defining $\mathbf{T}^{(0)}(z)=\left(  \Psi\left(  \mathbf{I}_{N}\right)
-z\mathbf{I}_{ML}\right)  ^{-1}$, and use the iterative definition
\begin{align}
\widetilde{\mathbf{T}}^{(p)}(z)  &  =-\frac{1}{z}\left(  \mathbf{I}_{N}
+c_{N}\overline{\Psi}^{T}\left(  \mathbf{T}^{(p)}(z)\right)  \right)
^{-1}\label{eq:seqTp}\\
\mathbf{T}^{(p+1)}(z)  &  =-\frac{1}{z}\left(  \mathbf{I}_{ML}+\Psi\left(
\widetilde{\mathbf{T}}^{(p)}(z)^{T}\right)  \right)  ^{-1}
\label{eq:seqTtildep}
\end{align}
for $p\geq0$. \ By Proposition \ref{prop:iteration} we see that the $L\times
L$ diagonal blocks of $\mathbf{T}^{(p)}(z)$ belong to the class $\mathcal{S}
_{L}\left(  \mathbb{R}^{\mathbb{+}}\right)  $, whereas $\widetilde{\mathbf{T}
}^{(p)}(z)$ belong to $\mathcal{S}_{N}\left(  \mathbb{R}^{\mathbb{+}}\right)
$. In order to prove the existence of a solution to the canonical equation, we
will first prove that the sequence $\mathbf{T}^{(p)}(z)$ has a limit in the
set of $ML\times ML$ diagonal block matrices with blocks belonging to the
class $\mathcal{S}_{L}\left(  \mathbb{R}^{\mathbb{+}}\right)  $. Then, in a
second step, we will prove that this limit is a solution to the canonical equation.

Our first objective is to show that, for $z$ belonging to a certain open
subset of $\mathbb{C}^{+}$, we have $\left\Vert \mathbf{T}^{(p+1)}
(z)-\mathbf{T}^{(p)}(z)\right\Vert \leq K(z)\left\Vert \mathbf{T}
^{(p)}(z)-\mathbf{T}^{(p-1)}(z)\right\Vert $, $p\geq0$, where $0<K(z)<1$. This
will show that, for each $z$ in this open subset, $\mathbf{T}^{(p)}(z)$ forms
a Cauchy sequence and therefore has a limit. Using Montel's theorem, we can
establish that convergence is uniform in compact sets and that the limiting
matrix function is analytic on $\mathbb{C}^{+}$. Define, for $p\geq1$ the
error matrices
\begin{align*}
\mathbf{\epsilon}^{(p)}(z) &  =\mathbf{T}^{(p)}(z)-\mathbf{T}^{(p-1)}(z)\\
\widetilde{\mathbf{\epsilon}}^{(p)}(z) &  =\widetilde{\mathbf{T}}
^{(p)}(z)-\widetilde{\mathbf{T}}^{(p-1)}(z)
\end{align*}
and note that we have a recurrent relationship through the identity
$\mathbf{A}^{-1}-\mathbf{B}^{-1}=\mathbf{A}^{-1}(\mathbf{B}-\mathbf{A}
)\mathbf{B}^{-1}$, namely
\begin{align*}
\mathbf{\epsilon}^{(p+1)}(z) &  =\Theta_{p}\left(  \mathbf{\epsilon}
^{(p)}(z)\right)  \\
\widetilde{\mathbf{\epsilon}}^{(p+1)}(z) &  =\widetilde{\Theta}_{p}\left(
\widetilde{\mathbf{\epsilon}}^{(p)}(z)\right)
\end{align*}
where we have defined the operators
\begin{align*}
\Theta_{p}\left(  \mathbf{X}\right)   &  =c_{N}z^{2}\mathbf{T}^{(p+1)}
(z)\Psi\left(  \widetilde{\mathbf{T}}^{(p)}(z)\overline{\Psi}\left(
\mathbf{X}\right)  \widetilde{\mathbf{T}}^{(p-1)}(z)\right)  \mathbf{T}
^{(p)}(z)\\
\widetilde{\Theta}_{p}\left(  \mathbf{X}\right)   &  =c_{N}z^{2}
\widetilde{\mathbf{T}}^{(p)}(z)\overline{\Psi}^{T}\left(  \mathbf{T}
^{(p+1)}(z)\Psi\left(  \mathbf{X}^{T})\right)  \mathbf{T}^{(p)}(z)\right)
\widetilde{\mathbf{T}}^{(p+1)}(z)
\end{align*}
for $p\geq1$.\ Using the properties of the operators we can obviously
establish that
\[
\max\left\{  \left\Vert \Theta_{p}\left(  \mathbf{X}\right)  \right\Vert
,\left\Vert \widetilde{\Theta}_{p}\left(  \mathbf{X}\right)  \right\Vert
\right\}  \leq\sup_{N}c_{N}\sup_{m,M,\nu}\left\vert \mathcal{S}_{m}\left(
\nu\right)  \right\vert ^{2}\frac{\left\vert z\right\vert ^{2}}{\left(
\mathrm{Im}z\right)  ^{4}}\left\Vert \mathbf{X}\right\Vert .
\]
Consider the domain
\[
\mathcal{D}=\left\{  z\in\mathbb{C}^{+}:\sup_{N}c_{N}\sup_{m,M,\nu}\left\vert
\mathcal{S}_{m}\left(  \nu\right)  \right\vert ^{2}\frac{\left\vert
z\right\vert ^{2}}{\left(  \mathrm{Im}z\right)  ^{4}}<\frac{1}{2}\right\}  .
\]
For $z\in\mathbb{C}^{+}$ we clearly see that both $\Theta_{p}\left(
\mathbf{X}\right)  $ and $\widetilde{\Theta}_{p}\left(  \mathbf{X}\right)  $
are contractive and therefore the sequences $\left(  \mathbf{T}^{(p)}
(z)\right)  _{p}$ and $\left(  \widetilde{\mathbf{T}}^{(p)}(z)\right)  _{p}$
are both Cauchy and have limits, which will be denoted by $\mathbf{T}(z)$ and
$\widetilde{\mathbf{T}}(z)$. Since the sequences $\left(  \mathbf{T}
^{(p)}(z)\right)  _{p}$ and $\left(  \widetilde{\mathbf{T}}^{(p)}(z)\right)
_{p}$ are uniformly bounded on compact subsets of $\mathbb{C}\backslash
\mathbb{R}^{+}$ (because they belong to $\mathcal{S}_{ML}\left(
\mathbb{R}^{\mathbb{+}}\right)  $ and $\mathcal{S}_{N}\left(  \mathbb{R}
^{\mathbb{+}}\right)  $ respectively), Montel's theorem establishes that
$\mathbf{T}(z)$ and $\widetilde{\mathbf{T}}(z)$ are analytic on $\mathbb{C}
\backslash\mathbb{R}^{+}$.\textbf{ }

It remains to prove that $\mathbf{T}(z)$ and $\widetilde{\mathbf{T}}(z)$
respectively belong to $\mathcal{S}_{ML}\left(  \mathbb{R}^{\mathbb{+}
}\right)  $ and $\mathcal{S}_{N}\left(  \mathbb{R}^{\mathbb{+}}\right)  $ and
that they satisfy the canonical system of equations. From the fact that
$\mathrm{Im}\mathbf{T}^{(p)}(z)\geq0$, $\mathrm{Im}z\mathbf{T}^{(p)}(z)\geq0$
and $\mathbf{T}^{(p)}(z)\mathbf{T}^{(p)}(z)^{H}\leq\delta_{z}^{-1}\mathbf{I}$
for $p\geq1$ we have $\mathrm{Im}\mathbf{T}(z)\geq0$, $\mathrm{Im}
z\mathbf{T}(z)\geq0$ and $\mathbf{T}(z)\mathbf{T}(z)^{H}\leq\delta_{z}
^{-1}\mathbf{I}$. The same argument applies to $\widetilde{\mathbf{T}}(z)$. On
the other hand, using the reasoning in the proof of Lemma
\ref{lemma:invertibility_RRtilde} we clearly see that both $\left(
\mathbf{I}_{ML}+\Psi\left(  \widetilde{\mathbf{T}}^{T}(z)\right)  \right)  $
and $\left(  \mathbf{I}_{N}+c_{N}\overline{\Psi}^{T}\left(  \mathbf{T}
(z)\right)  \right)  $ are invertible for $z\in\mathbb{C}\backslash
\mathbb{R}^{\mathbb{+}}$, and that they are the limits of the corresponding
terms on right hand side of (\ref{eq:seqTp})-(\ref{eq:seqTtildep}). This shows
that the pair $\mathbf{T}(z)$, $\widetilde{\mathbf{T}}(z)$ satisfies the
canonical system of equations. Noting that
\[
\lim_{y\rightarrow\infty}-\mathrm{i}y\mathbf{T}(\mathrm{i}y)=\lim
_{y\rightarrow\infty}\left(  \mathbf{I}_{ML}+\Psi\left(  \widetilde
{\mathbf{T}}^{T}(\mathrm{i}y)\right)  \right)  ^{-1}=\mathbf{I}_{ML}
\]
(because $\left\Vert \widetilde{\mathbf{T}}(z)\right\Vert <\delta_{z}
^{-1}\mathbf{I}$) \ we can conclude $\mathbf{T}(z)$ belongs to $\mathcal{S}
_{ML}\left(  \mathbb{R}^{\mathbb{+}}\right)  $. A similar reasoning shows that
$\widetilde{\mathbf{T}}(z)$ belongs to $\mathcal{S}_{N}\left(  \mathbb{R}
^{\mathbb{+}}\right)  $. \newline

Following the proof of (\ref{eq:lower-bound-RR*},
\ref{eq:lower-bound-tildeRtildeR*}), it is easy to check that each solution of
(\ref{eq:canonical-T}, \ref{eq:canonical-tildeT}) satisfies
(\ref{eq:lower-bound-TT*}, \ref{eq:lower-bound-tildeTtildeT*}). \newline

Let us now prove unicity. For this, it would be possible to use arguments
based on the analyticity of the solutions and the Montel theorem as in the
existence proof. We however prefer to use a different approach because the
corresponding ideas will be used later, and rather prove that for each $z$,
system (\ref{eq:canonical-T}, \ref{eq:canonical-tildeT}) considered as a
system in the set of $ML\times ML$ and $N\times N$ matrices, has a unique
solution. We fix $z\in\mathbb{C}\setminus\mathbb{R}^{+}$, and assume that
$\mathbf{T}(z),\widetilde{\mathbf{T}}(z)$ and $\mathbf{S}(z),\widetilde
{\mathbf{S}}(z)$ are matrices that are solutions of the system
(\ref{eq:canonical-T}, \ref{eq:canonical-tildeT}) of equations at point $z$.
It is easily seen that
\begin{equation}
\mathbf{T}(z)-\mathbf{S}(z)=c_{N}z^{2}\mathbf{S}(z)\Psi\left(  \widetilde
{\mathbf{S}}^{T}(z)\overline{\Psi}\left(  \mathbf{T}(z)-\mathbf{S}(z)\right)
\widetilde{\mathbf{T}}^{T}(z)\right)  \mathbf{T}(z)\label{eq:T-S}
\end{equation}
The above equation can alternatively be written as
\[
\mathbf{T}(z)-\mathbf{S}(z)=\Phi_{0}\left(  \mathbf{T}(z)-\mathbf{S}
(z)\right)
\]
where we have defined the operator $\Phi_{0}\left(  \mathbf{X}\right)  $ as
\begin{equation}
\Phi_{0}\left(  \mathbf{X}\right)  =c_{N}z^{2}\mathbf{S}(z)\Psi\left(
\widetilde{\mathbf{S}}^{T}(z)\overline{\Psi}\left(  \mathbf{X}\right)
\widetilde{\mathbf{T}}^{T}(z)\right)  \mathbf{T}(z)\label{eq:def-Phi0}
\end{equation}
where $\mathbf{X}$ is an $ML\times ML$\ block-diagonal matrix. We note that
operator ${\Phi}_{0}$ depends on point $z$, but we do not mention this
dependency in order to simplify the notations. Our objective is to show that
the equation $\Phi_{0}\left(  \mathbf{X}\right)  =\mathbf{X}$ accepts a unique
solution in the set of block-diagonal matrices, which is trivially given by
$\mathbf{X=0}$. This will imply that $\mathbf{T}(z)=\mathbf{S}(z)$,
contradicting the original hypothesis.

We iteratively define $\Phi_{0}^{(n)}\left(  \mathbf{X}\right)  =\Phi
_{0}\left(  \Phi_{0}^{(n-1)}\mathbf{X}\right)  $ for $n\in\mathbb{N}$, with
$\Phi_{0}^{(1)}\left(  \mathbf{X}\right)  =\Phi_{0}\left(  \mathbf{X}\right)
$. Let $\Phi_{0}^{(n)}\left(  \mathbf{X}\right)  ^{m,m}$ denote the $L\times
L$ sized $m$th diagonal block of $\Phi_{0}^{(n)}\left(  \mathbf{X}\right)  $.
In the following, we establish that for block diagonal $ML\times ML$ matrix
$\mathbf{X}$, it holds that $\lim_{n\rightarrow+\infty}\Phi_{0}^{(n)}\left(
\mathbf{X}\right)  =0$. If this property holds, a solution of the equation
$\mathbf{X}=\Phi_{0}(\mathbf{X})$ satisfies $\mathbf{X}=\Phi_{0}
^{(n)}(\mathbf{X})$ for each $n$, thus leading to $\mathbf{X}=0$. It is useful
to mention that in the following analysis, dimensions $L,M,N$ are fixed. We
establish the following Proposition, which, of course, implies that each
element of matrix $\Phi_{0}^{(n)}\left(  \mathbf{X}\right)  $ converges
towards $0$, i.e. that matrix $\Phi_{0}^{(n)}\left(  \mathbf{X}\right)  $
converges towards $0$.

\begin{prop}
\label{prop:normPhi0} For each $\,m=1,\ldots,M$, and for each $L$--dimensional
vectors $\mathbf{a}$ and $\mathbf{b}$, it holds that
\[
\left\vert \mathbf{a}^{H}\Phi_{0}^{(n)}\left(  \mathbf{X}\right)
^{m,m}\mathbf{b}\right\vert \rightarrow0
\]
as $n\rightarrow\infty$ for each $ML \times ML$  matrix $\mathbf{X}$.
\end{prop}

To simplify the notation, we drop the dependence on $z$ from $\mathbf{T}(z)$,
$\widetilde{\mathbf{T}}(z)$, $\mathbf{S}(z)$ and $\widetilde{\mathbf{S}}(z)$
in what follows. We begin by defining two operators $\Phi_{\mathbf{S}
}\left(  \mathbf{X}\right)  $ and $\Phi_{\mathbf{T}^{H}}\left(  \mathbf{X}\right)
$ that operate on $ML\times ML$ matrices as
\begin{align}
\Phi_{\mathbf{S}}\left(  \mathbf{X}\right)   &  =c_{N}\left\vert z\right\vert
^{2}\mathbf{S}\Psi\left(  \widetilde{\mathbf{S}}^{T}\overline{\Psi}\left(
\mathbf{X}\right)  \widetilde{\mathbf{S}}^{\ast}\right)  \mathbf{S}
^{H}\label{eq:defPhi_S}\\
\Phi_{\mathbf{T}^{H}}\left(  \mathbf{X}\right)   &  =c_{N}\left\vert
z\right\vert ^{2}\mathbf{T}^{H}\Psi\left(  \widetilde{\mathbf{T}}^{\ast
}\overline{\Psi}\left(  \mathbf{X}\right)  \widetilde{\mathbf{T}}^{T}\right)
\mathbf{T}.\label{eq:defPhi_TT}
\end{align}
We remark that $\Phi_{\mathbf{S}}(\mathbf{X})\geq0$ and $\Phi_{\mathbf{T}^{H}
}(\mathbf{X})\geq0$ if $\mathbf{X}\geq0$. Matrices $\mathbf{T},\widetilde
{\mathbf{T}},\mathbf{S},\widetilde{\mathbf{S}}$ being solutions of equations
(\ref{eq:canonical-T}, \ref{eq:canonical-tildeT}) are non singular. Therefore,
$\Phi_{\mathbf{S}}(\mathbf{X})$ and $\Phi_{\mathbf{T}^{H}}(\mathbf{X})$ are
also positive definite as soon as $\mathbf{X}$ is positive definite.

Consider the integral representation of the $m$th diagonal block of $\Phi
_{0}\left(  \mathbf{X}\right)  $, that is
\begin{align*}
\Phi_{0}\left(  \mathbf{X}\right)  ^{m,m} &  =c_{N}z^{2}\int_{0}^{1}\int
_{0}^{1}F_{m}(\nu,\alpha)\mathbf{S}^{m,m}\mathbf{d}_{L}\left(  \nu\right)
\mathbf{a}_{N}^{H}\left(  \nu\right)  \widetilde{\mathbf{S}}\mathbf{d}
_{N}\left(  \alpha\right)  \times\\
&  \times\mathbf{d}_{N}^{H}\left(  \alpha\right)  \widetilde{\mathbf{T}
}\mathbf{a}_{N}\left(  \nu\right)  \mathbf{d}_{L}^{H}\left(  \nu\right)
\mathbf{T}^{m,m}d\nu d\alpha
\end{align*}
where
\[
F_{m}(\nu,\alpha)=\mathcal{S}_{m}\left(  \nu\right)  \frac{1}{M}\sum
_{k=1}^{M}\mathcal{S}_{k}\left(  \alpha\right)  \mathbf{a}_{L}^{H}\left(
\alpha\right)  \mathbf{X}^{k,k}\mathbf{a}_{L}\left(  \alpha\right)  .
\]
It turns out that, for each integer $n\geq0$ and each $m=1,\ldots,M$ we have
\begin{multline}
\Phi_{0}^{(n+1)}\left(  \mathbf{X}\right)  ^{m,m}\left(  \Phi_{\mathbf{T}^{H}
}^{(n+1)}\left(  \mathbf{I}\right)  ^{m,m}\right)  ^{-1}\left(  \Phi
_{0}^{(n+1)}\left(  \mathbf{X}\right)  ^{m,m}\right)  ^{H}\leq
\label{eq:boundPhi0PhiTPhiS}\\
\leq\Phi_{\mathbf{S}}\left(  \Phi_{0}^{(n)}\left(  \mathbf{X}\right)  \left(
\Phi_{\mathbf{T}^{H}}^{(n)}\left(  \mathbf{I}\right)  \right)  ^{-1}\Phi
_{0}^{(n)}\left(  \mathbf{X}\right)  ^{H}\right)  ^{m,m}
\end{multline}
To see this, we consider the matrix
\begin{equation}
\mathcal{M}_{m}(\mathbf{X},\mathbf{B})=\left[
\begin{array}
[c]{cc}
\Phi_{\mathbf{S}}\left(  \mathbf{XB}^{-1}\mathbf{X}^{H}\right)  ^{m,m} &
\Phi_{0}\left(  \mathbf{X}\right)  ^{m,m}\\
\Phi_{0}^{H}\left(  \mathbf{X}\right)  ^{m,m} & \Phi_{\mathbf{T}^{H}}\left(
\mathbf{B}\right)  ^{m,m}
\end{array}
\right]  \label{eq:MTX>0}
\end{equation}
where $\mathbf{B}$ is an arbitrary $ML\times ML$ Hermitian positive definite
block-diagonal matrix. It turns out that $\mathcal{M}_{m}(\mathbf{X}
,\mathbf{B})\geq0$. Indeed, to see this we only need to observe that this
matrix can alternatively be expressed as
\[
\mathcal{M}_{m}(\mathbf{X},\mathbf{B})=\frac{c_{N}}{M}\sum_{k=1}^{M}\int
_{0}^{1}\int_{0}^{1}\mathcal{S}_{m}\left(  \nu\right)  \mathcal{S}_{k}\left(
\alpha\right)  \Psi_{m,k}(\mathbf{X},\mathbf{B})\Psi_{m,k}(\mathbf{X}
,\mathbf{B})^{H}d\nu d\alpha
\]
where
\[
\Psi_{m,k}(\mathbf{X},\mathbf{B})=\left[
\begin{array}
[c]{c}
z\mathbf{S}^{m,m}\mathbf{d}_{L}\left(  \nu\right)  \mathbf{a}_{N}^{H}\left(
\nu\right)  \widetilde{\mathbf{S}}^{T}\mathbf{d}_{N}\left(  \alpha\right)
\mathbf{a}_{L}^{H}\left(  \alpha\right)  \mathbf{X}^{k,k}\left(
\mathbf{B}^{k,k}\right)  ^{-1/2}\\
z^{\ast}\left(  \mathbf{T}^{m,m}\right)  ^{H}\mathbf{d}_{L}\left(  \nu\right)
\mathbf{a}_{N}^{H}\left(  \nu\right)  \widetilde{\mathbf{T}}^{\ast}
\mathbf{d}_{N}\left(  \alpha\right)  \mathbf{a}_{L}^{H}\left(  \alpha\right)
\left(  \mathbf{B}^{k,k}\right)  ^{1/2}
\end{array}
\right]  .
\]
Now, since $\mathcal{M}_{m}(\mathbf{X},\mathbf{B})\geq0$, the Schur complement
of this matrix will also be positive semidefinite, so that we can state
\[
\Phi_{0}\left(  \mathbf{X}\right)  ^{m,m}\left(  \Phi_{\mathbf{T}^{H}}\left(
\mathbf{B}\right)  ^{m,m}\right)  ^{-1}\Phi_{0}^{H}\left(  \mathbf{X}\right)
^{m,m}\leq\Phi_{\mathbf{S}}\left(  \mathbf{XB}^{-1}\mathbf{X}^{H}\right)
^{m,m}.
\]
Thus, fixing $\mathbf{B}=\Phi_{\mathbf{T}^{H}}^{(n)}\left(  \mathbf{I}\right)
$\ and replacing $\mathbf{X}$ with $\Phi_{0}^{(n)}\left(  \mathbf{X}\right)  $
in the above equation, we directly obtain (\ref{eq:boundPhi0PhiTPhiS}).

By iterating the inequality in (\ref{eq:boundPhi0PhiTPhiS}) and using the
positivity of the operator $\Phi_{\mathbf{S}}\left(  \text{\textperiodcentered
}\right)  $ we obtain
\[
\Phi_{0}^{(n)}\left(  \mathbf{X}\right)  ^{m,m}\left(  \Phi_{\mathbf{T}^{H}
}^{(n)}\left(  \mathbf{I}\right)  ^{m,m}\right)  ^{-1}\left(  \Phi_{0}
^{(n)}\left(  \mathbf{X}\right)  ^{m,m}\right)  ^{H}\leq\Phi_{\mathbf{S}
}^{(n)}\left(  \mathbf{XX}^{H}\right)  ^{m,m}
\]
and
\begin{equation}
\Phi_{0}^{(n)}\left(  \mathbf{X}\right)  \left(  \Phi_{\mathbf{T}^{H}}
^{(n)}\left(  \mathbf{I}\right)  \right)  ^{-1}\left(  \Phi_{0}^{(n)}\left(
\mathbf{X}\right)  \right)  ^{H}\leq\Phi_{\mathbf{S}}^{(n)}\left(
\mathbf{XX}^{H}\right)  .\label{eq:fundamental-inequality-Phi0}
\end{equation}

We can now finalize the proof of Proposition \ref{prop:normPhi0} by noting
that, by the Cauchy-Schwarz inequality and the above inequality
\begin{align*}
&  \left\vert \mathbf{a}^{H}\Phi_{0}^{(n)}\left(  \mathbf{X}\right)
^{m,m}\mathbf{b}\right\vert \\
&  \leq\left[  \mathbf{a}^{H}\Phi_{0}^{(n)}\left(  \mathbf{X}\right)
^{m,m}\left(  \Phi_{\mathbf{T}^{H}}^{(n)}\left(  \mathbf{I}\right)
^{m,m}\right)  ^{-1}\left(  \Phi_{0}^{(n)}\left(  \mathbf{X}\right)
^{m,m}\right)  ^{H}\mathbf{a}\right]  ^{1/2}\left(  \mathbf{b}^{H}\left(
\Phi_{\mathbf{T}^{H}}^{(n)}\left(  \mathbf{I}\right)  ^{m,m}\right)
\mathbf{b}\right)  ^{1/2}\\
&  \leq\left[  \mathbf{a}^{H}\Phi_{\mathbf{S}}^{(n)}\left(  \mathbf{XX}
^{H}\right)  ^{m,m}\mathbf{a}\right]  ^{1/2}\left[  \mathbf{b}^{H}
\Phi_{\mathbf{T}^{H}}^{(n)}\left(  \mathbf{I}\right)  ^{m,m}\mathbf{b}\right]
^{1/2}.
\end{align*}
Therefore, to conclude the proof we only need to show that both $\Phi
_{\mathbf{S}}^{(n)}\left(  \mathbf{XX}^{H}\right)  ^{m,m}$ and $\Phi
_{\mathbf{T}^{H}}^{(n)}\left(  \mathbf{I}_{N}\right)  ^{m,m}$ converge to zero
as $n\rightarrow\infty$. This can be shown following the steps in
\cite{loubaton16}, as established in the following proposition.

\begin{lem}
Let $\mathbf{T}(z),\widetilde{\mathbf{T}}(z)$ be a solution to the canonical
equation (\ref{eq:canonical-T}, \ref{eq:canonical-tildeT}) at point $z$, and
let $\Phi_{\mathbf{T}}^{(n)}\left(  \mathbf{B}\right)  $ be defined, for a
positive semidefinite $\mathbf{B}$, as in (\ref{eq:defPhi_S}). Then, it holds
that
\begin{equation}
\Phi_{\mathbf{T}}^{(n)}\left(  \mathbf{B}\right)  \rightarrow
0\label{eq:convergence-Phi_Tn-zero}
\end{equation}
and
\begin{equation}
\Phi_{\mathbf{T}^{H}}^{(n)}\left(  \mathbf{B}\right)  \rightarrow
0\label{eq:convergence-Phi_THn-zero}
\end{equation}
as $n\rightarrow\infty$. Moreover, the series $\sum_{n=0}^{+\infty}
\Phi_{\mathbf{T}}^{(n)}\left(  \mathbf{B}\right)  $ and $\sum_{n=0}^{+\infty
}\Phi_{\mathbf{T}^{H}}^{(n)}\left(  \mathbf{B}\right)  $ are finite.
\end{lem}

\begin{proof}
If $\mathbf{T}$ is a solution to the canonical equation, we must have
\begin{equation}
\mathrm{Im}\mathbf{T}=\frac{\mathbf{T-T}^{H}}{2\mathrm{i}}=\mathrm{Im}
z\mathbf{TT}^{H}+\Phi_{\mathbf{T}}\left(  \mathrm{Im}\mathbf{T}\right)
\label{eq:ImT}
\end{equation}
or equivalently
\begin{equation}
\frac{\mathrm{Im}\mathbf{T}}{\mathrm{Im}z}=\mathbf{TT}^{H}+\Phi_{\mathbf{T}
}\left(  \frac{\mathrm{Im}\mathbf{T}}{\mathrm{Im}z}\right)  \label{eq:ImT}
\end{equation}
if $\mathrm{Im}(z)\neq0$. If $z$ belongs to $\mathbb{R}^{-\ast}$, then
$\frac{\mathrm{Im}\mathbf{T}}{\mathrm{Im}z}$ should be interpreted as positive
matrix $\mathbf{T}^{\prime}(z)=\int_{\mathbb{R}^{+}}\frac{d{\boldsymbol\mu
}(\lambda)}{(\lambda-z)^{2}}$, and the reader may verify that the following
arguments are still valid. Iterating the above relationship, we see that for
any $n\in\mathbb{N}$
\[
\frac{\mathrm{Im}\mathbf{T}}{\mathrm{Im}z}=\sum_{k=0}^{n}\Phi_{\mathbf{T}
}^{(k)}\left(  \mathbf{TT}^{H}\right)  +\Phi_{\mathbf{T}}^{(n+1)}\left(
\frac{\mathrm{Im}\mathbf{T}}{\mathrm{Im}z}\right)  .
\]
Since $\frac{\mathrm{Im}\mathbf{T}}{\mathrm{Im}z}\geq0$, we have
$\Phi_{\mathbf{T}}^{(n+1)}\left(  \frac{\mathrm{Im}\mathbf{T}}{\mathrm{Im}
z}\right)  \geq0$. On the other hand, we also have $\Phi_{\mathbf{T}}
^{(k)}\left(  \mathbf{TT}^{H}\right)  \geq0$ and therefore it holds that for
each $n$,
\[
\sum_{k=0}^{n}\Phi_{\mathbf{T}}^{(k)}\left(  \mathbf{TT}^{H}\right)  \leq
\frac{\mathrm{Im}\mathbf{T}}{\mathrm{Im}z}
\]
The series $\sum_{k=0}^{+\infty}\Phi_{\mathbf{T}}^{(k)}\left(  \mathbf{TT}
^{H}\right)  $ is thus convergent and we must have $\Phi_{\mathbf{T}}
^{(n)}\left(  \mathbf{TT}^{H}\right)  \rightarrow\mathbf{0}$ as $n\rightarrow
\infty$. Since matrix $\mathbf{T}$ is invertible, $\mathbf{T}^{H}
\mathbf{T}>\alpha(z)\mathbf{I}$ where $\alpha\left(  z\right)  >0$. Therefore,
$\alpha(z)\,\Phi_{\mathbf{T}}^{(n)}\left(  \mathbf{I}\right)  \leq
\Phi_{\mathbf{T}}^{(n)}\left(  \mathbf{TT}^{H}\right)  $, which implies that
$\Phi_{\mathbf{T}}^{(n)}\left(  \mathbf{I}\right)  $ converges towards $0$ and
that $\sum_{n=0}^{+\infty}\Phi_{\mathbf{T}}^{(n)}\left(  \mathbf{I}\right)
<+\infty$. Now, consider a general positive semidefinite $\mathbf{B}$. Then,
$\mathbf{B}\leq\Vert\mathbf{B}\Vert\mathbf{I}$ and $\Phi_{\mathbf{T}}
^{(n)}\left(  \mathbf{B}\right)  \leq\Vert\mathbf{B}\Vert\Phi_{\mathbf{T}
}^{(n)}\left(  \mathbf{I}\right)  $. Hence, it holds that $\Phi_{\mathbf{T}
}^{(n)}\left(  \mathbf{B}\right)  \rightarrow0$ and $\sum_{n=0}^{+\infty}
\Phi_{\mathbf{T}}^{(n)}\left(  \mathbf{B}\right)  <+\infty$. In particular,
$\Phi_{\mathbf{T}}^{(n+1)}\left(  \frac{\mathrm{Im}\mathbf{T}}{\mathrm{Im}
z}\right)  \rightarrow0$ as $n\rightarrow\infty$ and
\[
\frac{\mathrm{Im}\mathbf{T}}{\mathrm{Im}z}=\sum_{k=0}^{+\infty}\Phi
_{\mathbf{T}}^{(k)}\left(  \mathbf{TT}^{H}\right)  .
\]
\newline In order to establish (\ref{eq:convergence-Phi_THn-zero}), we use the
observation that
\[
\frac{\mathrm{Im}\mathbf{T}}{\mathrm{Im}z}=\mathbf{T}^{H}\mathbf{T}
+\Phi_{\mathbf{T}^{H}}\left(  \frac{\mathrm{Im}\mathbf{T}}{\mathrm{Im}
z}\right)
\]
and use the same arguments as above.
\end{proof}

\begin{rem}
\label{re:alpha} In the above analysis, $L,M,N$ are fixed parameters.
Therefore, $\alpha(z)$ a priori depends on $L,M,N$ as well as the norms of the
series $\sum_{n=0}^{+\infty}\Phi_{\mathbf{T}}^{(n)}\left(  \mathbf{I}\right)
$. In the following, a more precise analysis will be needed, and it will be
important to show that such an $\alpha(z)$ can be chosen independent from
$L,M,N$, and that the
\[
\sup_{N}\Vert\sum_{n=0}^{+\infty}\Phi_{\mathbf{T}}^{(n)}\left(  \mathbf{I}
\right)  \Vert<+\infty.
\]

\end{rem}

\section{Convergence towards the deterministic equivalent.}

\label{sec:convergence-R-T} If $(\mathbf{A}_{N})_{N\geq1}$ is a sequence of
deterministic uniformly bounded $ML\times ML$ matrices, Lemma \ref{le:variance-trace} implies  that the rate of convergence of
$\mathrm{var}\left(  \frac{1}{ML}\mathrm{tr}\left[  \mathbf{A}_N\mathbf{Q}
(z)\right]  \right)  $ is $\mathcal{O}(\frac{1}{MN})$. In the absence of extra
assumptions on $M$ (e.g. $M=\mathcal{O}(N^{\epsilon})$ for $\epsilon>0$), this
does not allow to conclude that
\begin{equation}
\frac{1}{ML}\mathrm{tr}\left(  \mathbf{A}_N\left[  \mathbf{Q}(z)-\mathbb{E}
(\mathbf{Q}(z))\right]  \right)  \rightarrow
0\label{eq:almost-sure-convergence-trace}
\end{equation}
almost surely. In order to obtain the almost sure convergence, we use the
identity
\[
\mathbb{E}\left\vert \frac{1}{ML}\mathrm{tr}\left(  \mathbf{A}_N\mathbf{Q}
^{\circ}(z)\right)  \right\vert ^{4}=\left\vert \mathbb{E}\left(  \frac{1}
{ML}\mathrm{tr}(\mathbf{A}_N\mathbf{Q}^{\circ}(z))\right)  ^{2}\right\vert
^{2}+\mathrm{var}\left[  \left(  \frac{1}{ML}\mathrm{tr}(\mathbf{A}_N
\mathbf{Q}^{\circ}(z))\right)  ^{2}\right]  ,
\]
remark that
\[
\left\vert \mathbb{E}\left(  \frac{1}{ML}\mathrm{tr}(\mathbf{A}_N\mathbf{Q}
^{\circ}(z))\right)  ^{2}\right\vert ^{2}\leq\left[  \mathrm{var}\left(
\frac{1}{ML}\mathrm{tr}\left[  \mathbf{A}_N\mathbf{Q}(z)\right]  \right)
\right]  ^{2}\leq C(z)\frac{1}{(MN)^{2}},
\]
and establish using the Nash-Poincar\'{e} inequality that
\[
\mathrm{var}\left[  \left(  \frac{1}{ML}\mathrm{tr}(\mathbf{A}_N\mathbf{Q}
^{\circ}(z))\right)  ^{2}\right]  \leq C(z)\frac{1}{(MN)^{2}}
\]
(the proof is straighforward and thus omitted). Markov inequality and
Borel-Cantelli's Lemma immediately imply that
(\ref{eq:almost-sure-convergence-trace}) holds, and that
\[
\frac{1}{ML}\mathrm{tr}\left(  \mathbf{A}_N\left[  \mathbf{Q}(z)-\mathbb{E}
(\mathbf{Q}(z))\right]  \right)  =\mathcal{O}_{P}\left(  \frac{1}{\sqrt{MN}
}\right)  .
\]
In the following, we study the behaviour of $\frac{1}{ML}\mathrm{tr}\left[
\mathbf{A}_N\left(\mathbb{E}(\mathbf{Q}(z)) - {\bf T}(z)\right)\right]  $. In this section, we first
establish that for each sequence of deterministic uniformly bounded $ML\times
ML$ matrices $(\mathbf{A}_{N})_{N\geq1}$, it holds that
\begin{equation}
\frac{1}{ML}\mathrm{tr}\left[  \mathbf{A}_N\left(  \mathbb{E}\mathbf{Q}
(z)-\mathbf{T}(z)\right)  \right]  \rightarrow
0\label{eq:convergence-E(Q)-T-towards-0}
\end{equation}
for each $z\in\mathbb{C}\setminus\mathbb{R}^{+}$, a property which, by virtue
of Proposition \ref{prop:control-Delta}, is equivalent to
\begin{equation}
\frac{1}{ML}\mathrm{tr}\left[  \mathbf{A}_N\left(  \mathbf{R}(z)-\mathbf{T}
(z)\right)  \right]  \rightarrow0\label{eq:convergence-R-T-towards-0}
\end{equation}
for each $z\in\mathbb{C}\setminus\mathbb{R}^{+}$. However,
(\ref{eq:convergence-E(Q)-T-towards-0}) does not provide any information on
the rate of convergence. Under the extra-assumption that
\begin{equation}
\frac{L^{3/2}}{MN}\rightarrow0\label{eq:extra-assumption-L-M-N}
\end{equation}
we establish that
\begin{equation}
\left\vert \frac{1}{ML}\mathrm{tr}\left[  \mathbf{A}_N\left(  \mathbb{E}
(\mathbf{Q}(z))-\mathbf{T}(z)\right)  \right]  \right\vert \leq C(z)\frac
{L}{MN}\label{eq:rate-convergence-E(Q)-T-towards-0}
\end{equation}
or equivalently that
\begin{equation}
\left\vert \frac{1}{ML}\mathrm{tr}\left[  \mathbf{A}_N\left(  \mathbf{R}
(z)-\mathbf{T}(z)\right)  \right]  \right\vert \leq C(z)\frac{L}
{MN}\label{eq:rate-convergence-R-T-towards-0}
\end{equation}
when $z$ belongs to a set $E_{N}$ defined by
\[
E_{N}=\{z\in\mathbb{C}\setminus\mathbb{R}^{+},\frac{L^{3/2}}{MN}
P_{1}(|z|)P_{2}(1/\delta_{z})<1\}
\]
where $P_{1}$ and $P_{2}$ are some nice polynomials. When
(\ref{eq:extra-assumption-L-M-N}) holds, each element $z\in\mathbb{C}
\setminus\mathbb{R}^{+}$ belongs to $E_{N}$ for $N$ greater than a certain
integer depending on $z$. Therefore, (\ref{eq:rate-convergence-R-T-towards-0})
implies that the rate of convergence of $\frac{1}{ML}\mathrm{tr}\left[
\mathbf{A}_N\left(  \mathbb{E}(\mathbf{Q}(z))-\mathbf{T}(z)\right)  \right]  $
towards $0$ is $\mathcal{O}(\frac{L}{MN})$ for each $z\in\mathbb{C}
\setminus\mathbb{R}^{+}$.

\subsection{Proof of (\ref{eq:convergence-R-T-towards-0}).}

In order to simplify the notations, matrix ${\bf A}_N$ is denoted by ${\bf A}$. Writing $\mathbf{R}(z)-\mathbf{T}(z)$ as $\mathbf{R}(z)\left(  \mathbf{R}
^{-1}(z)-\mathbf{T}^{-1}(z)\right)  \mathbf{T}(z)$ and {$\widetilde
{\mathbf{R}}$}$(z)-\widetilde{\mathbf{T}}(z)=\widetilde{\mathbf{T}}(z)\left(
\mathbf{T}^{-1}(z)-\mathbf{R}^{-1}(z)\right)  \widetilde{\mathbf{R}}(z)$, we
obtain immediately that
\begin{align}
\mathbf{R}(z)-\mathbf{T}(z) &  =z\mathbf{R}(z)\Psi\left(  \left(
\widetilde{\mathbf{R}}(z)-\widetilde{\mathbf{T}}(z)\right)  ^{T}\right)
\mathbf{T}(z)\label{eq:RminusT}\\
\left(  \widetilde{\mathbf{R}}(z)-\widetilde{\mathbf{T}}(z)\right)  ^{T} &
=zc_{N}\widetilde{\mathbf{R}}^{T}(z)\overline{\Psi}\left(  \mathbb{E}
\mathbf{Q}(z)-\mathbf{T}(z)\right)  \widetilde{\mathbf{T}}^{T}
(z).\label{eq:RtildeminusTtilde}
\end{align}
We introduce the linear operator {$\Phi$}$_{1}$ defined on the set of all
$ML\times ML$ matrices by
\begin{equation}
{\Phi}_{1}(\mathbf{X})=z^{2}\,c_{N}\,\mathbf{R}(z){\Psi}\left(  \widetilde
{\mathbf{R}}^{T}(z)\overline{{\Psi}}(\mathbf{X})\widetilde{\mathbf{T}}
^{T}(z)\right)  \mathbf{T}(z).\label{eq:def-operator-Phi1}
\end{equation}
The operator $\Phi_{1}$ is clearly obtained from operator $\Phi_{0}$ defined
by (\ref{eq:def-Phi0}) by replacing matrices $\mathbf{S}(z)$ and
$\widetilde{\mathbf{S}}(z)$ by matrices $\mathbf{R}(z)$ and $\widetilde
{\mathbf{R}}(z)$. Then, it holds that
\begin{equation}
\mathbf{R}(z)-\mathbf{T}(z)={\Phi}_{1}\left(  \mathbf{R}(z)-\mathbf{T}
(z)\right)  +{\Phi}_{1}\left(  {\boldsymbol\Delta}(z)\right)
.\label{eq:equation-R-T}
\end{equation}
Thus, matrix $\mathbf{R}(z)-\mathbf{T}(z)$ can be interpreted as the solution
of the linear equation (\ref{eq:equation-R-T}). Therefore, in some sense,
showing that $\mathbf{R}(z)-\mathbf{T}(z)$ converges towards $0$ can be proved
by showing that operator $\mathbf{I}-{\boldsymbol\Phi}_{1}$ is invertible, and
that the action of its inverse on ${\Phi}_{1}\left(  {\boldsymbol\Delta
}(z)\right)  $ still converges towards $0$. In this subsection, we implicitely
prove that ${\Phi}_{1}$ is a contractive operator for $z$ well chosen, obtain
that $\frac{1}{ML}\mathrm{tr}\left[  \mathbf{A}\left(  \mathbf{R}
(z)-\mathbf{T}(z)\right)  \right]  =\mathcal{O}(\frac{L}{MN})$ for such $z$,
and use Montel's theorem to conclude that (\ref{eq:convergence-R-T-towards-0})
holds for each $z\in\mathbb{C}\setminus\mathbb{R}^{+}$. \newline

Using (\ref{eq:transpose_ops_gen}), we remark that for each $ML\times ML$
matrices $\mathbf{A}$ and $\mathbf{B}$,
\begin{equation}
\frac{1}{ML}\mathrm{tr}({\Phi}_{1}(\mathbf{B})\mathbf{A}))=\frac{1}
{ML}\mathrm{tr}(\mathbf{B}{\Phi}_{1}^{t}(\mathbf{A}
))\label{eq:implicit-definition-adjoint}
\end{equation}
where ${\Phi}_{1}^{t}$ represents the linear operator defined on the $ML\times
ML$ matrices by
\begin{equation}
{\Phi}_{1}^{t}(\mathbf{A})=z^{2}c_{N}\,{\Psi}\left(  \widetilde{\mathbf{T}
}^{T}\overline{{\Psi}}(\mathbf{T}\mathbf{A}\mathbf{R})\widetilde{\mathbf{R}
}^{T}\right)  \label{eq:def-transpose-Phi1}
\end{equation}
We remark that operator ${\Phi}_{1}^{t}$ is related to the adjoint ${\Phi}
_{1}^{\ast}$ of ${\Phi}_{1}$ w.r.t. the scalar product $<\mathbf{A}
,\mathbf{B}>=\frac{1}{ML}\mathrm{tr}(\mathbf{A}\mathbf{B}^{\ast})$ through the
relation ${\Phi}_{1}^{t}(\mathbf{A})=\left(  {\Phi}_{1}^{\ast}(\mathbf{A}
^{H})\right)  ^{H}$.

If $\mathbf{A}$ is a  $ML\times
ML$ deterministic matrix, (\ref{eq:equation-R-T}) leads to
\begin{equation}
\frac{1}{ML}\mathrm{tr}\left[  \mathbf{A}\left(  \mathbf{R}(z)-\mathbf{T}
(z)\right)  \right]  =\frac{1}{ML}\mathrm{tr}\left[  {\Phi}_{1}^{t}
(\mathbf{A})\left(  \mathbf{R}(z)-\mathbf{T}(z)\right)  \right]  +\frac{1}
{ML}\mathrm{tr}\left[  {\Phi}_{1}^{t}(\mathbf{A})\;{\boldsymbol\Delta
}(z)\right]  \label{eq:expre-tr-R-T}
\end{equation}
Obviously, from the properties of the operators $\Psi($\textperiodcentered$)$
and $\overline{\Psi}($\textperiodcentered$)$ we can write
\begin{align*}
\left\Vert {\Phi}_{1}^{t}(\mathbf{A})\right\Vert  &  \leq\left\vert
z\right\vert ^{2}c_{N}\sup_{m,M,\nu}\left\vert \mathcal{S}_{m}\left(
\nu\right)  \right\vert ^{2}\left\Vert \mathbf{R}(z)\right\Vert \left\Vert
\widetilde{\mathbf{R}}(z)\right\Vert \left\Vert \mathbf{A}\right\Vert
\left\Vert \mathbf{T}(z)\right\Vert \left\Vert \widetilde{\mathbf{T}
}(z)\right\Vert \\
&  \leq\frac{\left\vert z\right\vert ^{2}}{\left(  \delta_{z}\right)  ^{4}
}c_{N}\sup_{m,M,\nu}\left\vert \mathcal{S}_{m}\left(  \nu\right)  \right\vert
^{2}\Vert\mathbf{A}\Vert.
\end{align*}
Let us consider the domain
\[
\mathcal{D}=\left\{  z\in\mathbb{C}\setminus\mathbb{R}^{+}:\frac{\left\vert
z\right\vert ^{2}}{\left(  \delta_{z}\right)  ^{4}}\sup_{N}c_{N}\sup_{m,M,\nu
}\left\vert \mathcal{S}_{m}\left(  \nu\right)  \right\vert ^{2}<\frac{1}
{2}\right\}  .
\]
and define $\alpha(z)$ as
\[
\alpha(z)=\sup_{\Vert\mathbf{B}\Vert\leq1}\left\vert \frac{1}{ML}
\mathrm{tr}\left(  (\mathbf{R}(z)-\mathbf{T}(z))\mathbf{B}\right)
\right\vert
\]
Then, we establish that if $z\in\mathcal{D}$, then $\alpha(z)\leq C(z)\frac
{L}{MN}$. For this, we consider $\mathbf{A}$ such that $\Vert\mathbf{A}
\Vert\leq1$. Using that
\[
\left\vert \frac{1}{ML}\mathrm{tr}\left[  {\Phi}_{1}^{t}(\mathbf{A})\left(
\mathbf{R}(z)-\mathbf{T}(z)\right)  \right]  \right\vert \leq\alpha
(z)\,\Vert{\Phi}_{1}^{t}(\mathbf{A})\Vert
\]
and that $\Vert{\Phi}_{1}(\mathbf{A})\Vert\leq1/2$ if $z\in\mathcal{D}$, we
deduce from (\ref{eq:expre-tr-R-T}) that
\[
\alpha(z)\leq\frac{\alpha(z)}{2}+\sup_{\Vert\mathbf{A}\Vert\leq1}\left\vert
\frac{1}{ML}\mathrm{tr}\left(  {\Phi}_{1}^{t}(\mathbf{A}){\boldsymbol\Delta
}(z)\right)  \right\vert
\]
As $\Vert{\Phi}_{1}^{t}(\mathbf{A})\Vert\leq1/2$ if $\Vert\mathbf{A}\Vert
\leq1$, (\ref{eq:control-trace-Delta}) implies that
\[
\sup_{\Vert\mathbf{A}\Vert\leq1}\left\vert \frac{1}{ML}\mathrm{tr}\left(
{\Phi}_{1}^{t}(\mathbf{A}){\boldsymbol\Delta}(z)\right)  \right\vert \leq
C(z)\frac{L}{MN}
\]
This implies that $\alpha(z)\leq C(z)\frac{L}{MN}$ for each $z\in\mathcal{D}$,
and that for each uniformly bounded sequence of $ML\times ML$ matrices
$\mathbf{A}_{N}$, (\ref{eq:convergence-R-T-towards-0}) holds on $\mathcal{D}$.
Montel's theorem immediately implies that (\ref{eq:convergence-R-T-towards-0})
also holds for each $z\in\mathbb{C}\setminus\mathbb{R}^{+}$.

\subsection{Proof of (\ref{eq:rate-convergence-E(Q)-T-towards-0}).}

We now establish (\ref{eq:rate-convergence-E(Q)-T-towards-0}) for each
$z\in\mathbb{C}\setminus\mathbb{R}^{+}$ under Assumption
(\ref{eq:extra-assumption-L-M-N}). For this, we establish that the linear
equation (\ref{eq:equation-R-T}) can be solved for each $z\in\mathbb{C}
\setminus\mathbb{R}^{+}$. For this, we first prove the following proposition.

\begin{prop}
\label{eq:convergence-series-Phi1} It exists 2 nice polynomials $P_{1}$ and
$P_{2}$ (see Definition \ref{def:nice}) such that for each $ML\times ML$ matrix $\mathbf{X}$, the series
$\sum_{n=0}^{\infty}{\Phi}_{1}^{(n)}(\mathbf{X})$ is convergent when $z$
belongs the set $E_{N}$ defined by
\begin{equation}
E_{N}=\{z\in\mathbb{C}\setminus\mathbb{R}^{+},\frac{L^{3/2}}{MN}
P_{1}(|z|)P_{2}(1/\delta_{z})<1\}\label{eq:def-EN}
\end{equation}

\end{prop}

In order to establish Proposition \ref{eq:convergence-series-Phi1}, we first
remark that for each matrix $\mathbf{X}$, it holds that
\begin{equation}
{\Phi}_{1}^{(n)}(\mathbf{X})\left(  {\Phi}_{\mathbf{T}^{H}}^{(n)}
(\mathbf{I})\right)  ^{-1}\left(  {\Phi}_{1}^{(n)}(\mathbf{X})\right)
^{H}\leq{\Phi}_{\mathbf{R}}^{(n)}(\mathbf{X}\mathbf{X}^{H}
)\label{eq:inequality-fundamental-Phi1}
\end{equation}
where ${\Phi}_{\mathbf{T}^{H}}($\textperiodcentered$)$ is defined in
(\ref{eq:defPhi_TT}) and where ${\Phi}_{\mathbf{R}}\left(
\text{\textperiodcentered}\right)  $ is as in (\ref{eq:defPhi_S}) replacing
$\mathbf{S}$ with $\mathbf{R}$.This inequality is proved in the same way as
(\ref{eq:fundamental-inequality-Phi0}). It has already been proved that
$\sum_{n=0}{\Phi}_{\mathbf{T}^{H}}^{(n)}(\mathbf{I})$ is convergent. Following
the proof of the uniqueness in Proposition \ref{proposition:existence_unicity}
, we will obtain the convergence of the series $\sum_{n=0}^{\infty}{\Phi}
_{1}^{(n)}(\mathbf{X})$, i.e. that
\[
\sum_{n=0}^{+\infty}\left\vert \mathbf{a}^{H}{\Phi}_{1}^{(n)}(\mathbf{X}
)\mathbf{b}\right\vert <+\infty
\]
if we are able to establish that
\begin{equation}
\sum_{n=0}^{+\infty}{\Phi}_{\mathbf{R}}^{(n)}(\mathbf{X}\mathbf{X}
^{H})<+\infty.\label{eq:convergence-series-Phi1n}
\end{equation}
When $\operatorname{Im}z\neq0$, we can write
\[
\frac{\mathrm{Im}(\mathbf{R})}{\mathrm{Im}(z)}=\mathbf{R}\mathbf{R}^{H}+{\Phi
}_{\mathbf{R}}\left(  \frac{\mathrm{Im}(\mathbb{E}(\mathbf{Q)})}
{\mathrm{Im}(z)}\right)
\]
Therefore, it holds that
\begin{equation}
\frac{\mathrm{Im}(\mathbb{E}(\mathbf{Q}))}{\mathrm{Im}(z)}=\mathbf{R}
\mathbf{R}^{H}+\frac{\mathrm{Im}({\boldsymbol\Delta})}{\mathrm{Im}(z)}+{\Phi
}_{\mathbf{R}}\left(  \frac{\mathrm{Im}(\mathbb{E}(\mathbf{Q)})}
{\mathrm{Im}(z)}\right)  .\label{eq:equation-ImE(Q)}
\end{equation}
When $z\in\mathbf{R}^{-\ast}$, we can still interpret $\frac{\mathrm{Im}
(\mathbf{R})}{\mathrm{Im}(z)}$ and $\frac{\mathrm{Im}(\mathbb{E}(\mathbf{Q}
))}{\mathrm{Im}(z)}$ as $\mathbf{R}^{\prime}(z)$ and $\mathbb{E}
(\mathbf{Q}^{\prime}(z))$ and the following reasoning holds as well. In order
to use the ideas of the proof of the uniqueness in Proposition
\ref{proposition:existence_unicity}, matrix $\mathbf{R}\mathbf{R}^{H}
+\frac{\mathrm{Im}({\boldsymbol\Delta})}{\mathrm{Im}(z)}$ should be positive.
By (\ref{eq:lower-bound-RR*}), matrix $\mathbf{R}\mathbf{R}^{H}$ verifies
$\mathbf{R}\mathbf{R}^{H}\geq\frac{\delta_{z}^{2}}{Q_{2}(|z|)}\mathbf{I}$ for
some nice polynomial $Q_{2}$. In order to guarantee the positivity of
$\mathbf{R}\mathbf{R}^{H}+\frac{\mathrm{Im}({\boldsymbol\Delta})}
{\mathrm{Im}(z)}$ on a large subset of $\mathbb{C} \setminus \mathbb{R}^{+}$,
condition $\Vert\frac{\mathrm{Im}({\boldsymbol\Delta})}{\mathrm{Im}(z)}
\Vert\rightarrow0$ should hold. However, it can be shown that the rate of
convergence of $\Vert\frac{\mathrm{Im}({\boldsymbol\Delta})}{\mathrm{Im}
(z)}\Vert$ is $\mathcal{O}\left(  (L/M^{3})^{1/2}\right)  $. Assuming that
$L/M^{3}\rightarrow0$ is a stronger condition than Assumption
(\ref{eq:extra-assumption-L-M-N}) which is equivalent to $L/M^{4}\rightarrow
0$. Therefore, we have to modify the proof of the uniqueness in Proposition
\ref{proposition:existence_unicity}. Instead of considering Eq.
(\ref{eq:equation-ImE(Q)}), we consider
\begin{equation}
{\Phi}_{\mathbf{R}}\left(  \frac{\mathrm{Im}(\mathbb{E}(\mathbf{Q}
))}{\mathrm{Im}(z)}\right)  ={\Phi}_{\mathbf{R}}\left(  \mathbf{R}
\mathbf{R}^{H}+\frac{\mathrm{Im}({\boldsymbol\Delta})}{\mathrm{Im}(z)}\right)
+{\Phi}_{\mathbf{R}}^{(2)}\left(  \frac{\mathrm{Im}(\mathbb{E}(\mathbf{Q)}
)}{\mathrm{Im}(z)}\right)  
\end{equation}
This time, we will see that $\Phi_{\mathbf{R}}\left(  \mathbf{R}
\mathbf{R}^{\ast}+\frac{\mathrm{Im}({\boldsymbol\Delta})}{\mathrm{Im}
(z)}\right)  \geq \frac{\delta_{z}^{2}}{Q_{2}(|z|)}\mathbf{I}$ when $z$
belongs to a set $E_{N}$ defined by (\ref{eq:def-EN}), and this property
allows to prove (\ref{eq:convergence-series-Phi1n}) if $z\in E_{N}$. In order
to establish this, we first state the following Lemma.

\begin{lem}
\label{le:convergence-rate-ImDelta-overImz} It holds that
\begin{equation}
\Vert\overline{{\Psi}}\left(  \frac{\mathrm{Im}({\boldsymbol\Delta}
)}{\mathrm{Im}(z)}\right)  \Vert\leq C(z)\frac{L^{3/2}}{MN}
\label{eq:convergence-rate-ImDelta-overImz}
\end{equation}

\end{lem}
\begin{proof}
Following the proof of Corollary \ref{cor:norm-psibar-delta}, it is sufficient
to establish that if $(\mathbf{b}_{1,N})_{N \geq 1}$ and $(\mathbf{b}
_{2,N})_{N \geq 1}$ are two sequences of $L$ dimensional vectors such
that $\sup_{N}\Vert\mathbf{b}_{i,N}\Vert<b<+\infty$ for $i=1,2$, and if
$((d_{m,N})_{m=1,\ldots, M})_{N\geq1}$ are deterministic complex number verifying
$\sup_{N,m}|d_{m,N}|<d<+\infty$, then, it holds that
\begin{equation}
\left\vert \mathbf{b}_{1,N}^{H}\left(  \frac{1}{M}\sum_{m=1}^{M}d_{m,N}
\frac{\mathrm{Im}\left(  {\boldsymbol\Delta}^{m,m}\right)  }{\mathrm{Im}
(z)}\right)  \mathbf{b}_{2,N}\right\vert \leq d\,b^{2}\,P_{1}(|z|)P_{2}
(1/\delta_{z})\frac{L^{3/2}}{MN}
\label{eq:controle-trace-partielle-ImDelta-overImz}
\end{equation}
for some nice polynomials $P_{1}$ and $P_{2}$.
(\ref{eq:controle-trace-partielle-ImDelta-overImz}) can be established by
adapting in a straightforward way the arguments of the proof of Lemma B-1 of
\cite{hachem-loubaton-najim-vallet-ihp-2013}. This completes the proof of the
Lemma.
\end{proof}

Using the definition of operator $\Phi_{\mathbf{R}}$ and the properties of
$\mathbf{R}$ and $\widetilde{\mathbf{R}}$, we deduce from Lemma
\ref{le:convergence-rate-ImDelta-overImz} that
\[
\left\Vert {\Phi}_{\mathbf{R}}\left(  \frac{\mathrm{Im}({\boldsymbol\Delta}
)}{\mathrm{Im}(z)}\right)  \right\Vert \leq C(z)\frac{L^{3/2}}{MN}
\]
when $z\in\mathbb{E}_{N}$. As $\mathbf{R}(z)\mathbf{R}^{H}(z)\geq\frac
{\delta_{z}^{2}}{Q_{2}(|z|)}\mathbf{I}_{ML}$, it holds that $\overline{{\Psi}
}(\mathbf{R}(z)\mathbf{R}^{H}(z))\geq\frac{\delta_{z}^{2}}{Q_{2}
(|z|)}\overline{{\Psi}}(\mathbf{I}_{ML})$. Since all spectral densities
$(\mathcal{S}_{m})$ verify (\ref{eq:lowerbound-S}), it is clear that
$\overline{{\Psi}}(\mathbf{I}_{ML})\geq C\;\mathbf{I}_{N}$ for some nice
constant $C$. Therefore, it appears that
\[
\overline{{\Psi}}(\mathbf{R}(z)\mathbf{R}^{H}(z))\geq C\,\frac{\delta_{z}^{2}
}{Q_{2}(|z|)}\mathbf{I}_{N}
\]
Lemma \ref{le:convergence-rate-ImDelta-overImz} thus implies that
\[
\overline{{\Psi}}\left(  \mathbf{R}(z)\mathbf{R}(z)^{H}+\frac{\mathrm{Im}
({\boldsymbol\Delta})}{\mathrm{Im}(z)}\right)  \geq C/2\,\frac{\delta_{z}^{2}
}{Q_{2}(|z|)}\mathbf{I}_{N}
\]
on the set $E_{N}$ defined by
\[
E_{N}=\{z\in\mathbb{C}\setminus\mathbb{R}^{+},C(z)\frac{L^{3/2}}{MN}\leq
C/2\frac{\delta_{z}^{2}}{Q_{2}(|z|)}\}
\]
which can be written as in (\ref{eq:def-EN}). Since matrix $\widetilde
{\mathbf{R}}^{T}(z)\widetilde{\mathbf{R}}^{\ast}(z)$ is also greater than
$\frac{\delta_{z}^{2}}{Q_{2}(|z|)}$ for some nice polynomial $Q_{2}$ (see
(\ref{eq:lower-bound-tildeRtildeR*})), we eventually obtain that
\begin{equation}
{\Phi}_{\mathbf{R}}\left(  \mathbf{R}(z)\mathbf{R}(z)^{H}+\frac{\mathrm{Im}
({\boldsymbol\Delta})}{\mathrm{Im}(z)}\right)  \geq\frac{\delta_{z}^{6}}
{Q_{2}(|z|)}\mathbf{I}_{ML}\label{eq:phiR-bounded-from-below}
\end{equation}
for each $z\in\mathbb{E}_{N}$. Using the same arguments than in the proof of
the uniqueness in Proposition \ref{proposition:existence_unicity}, we obtain
that ${\Phi}_{\mathbf{R}}^{(n)}\left(  \mathbf{R}(z)\mathbf{R}^{H}
(z)+\frac{\mathrm{Im}({\boldsymbol\Delta})}{\mathrm{Im}(z)}\right)
\rightarrow0$ when $n\rightarrow0$, and that
\[
\Phi_{\mathbf{R}}\left(  \frac{\mathrm{Im}(\mathbb{E}(\mathbf{Q}
))}{\mathrm{Im}(z)}\right)  =\sum_{n=0}^{+\infty}\Phi_{\mathbf{R}}
^{(n)}\left(  {\Phi}_{\mathbf{R}}\left(  \mathbf{R}(z)\mathbf{R}^{H}
(z)+\frac{\mathrm{Im}({\boldsymbol\Delta})}{\mathrm{Im}(z)}\right)  \right)
\]
when $z\in E_{N}$. Using (\ref{eq:phiR-bounded-from-below}) as well as
\[
\Phi_{\mathbf{R}}\left(  \frac{\mathrm{Im}(\mathbb{E}(\mathbf{Q}
))}{\mathrm{Im}(z)}\right)  \leq C(z)\;\mathbf{I}
\]
we eventually obtain that
\[
\sum_{n=0}^{+\infty}\Phi_{\mathbf{R}}^{(n)}(\mathbf{I})\leq C(z)\,\mathbf{I}
\]
when $z\in E_{N}$. Therefore, for each $ML\times ML$ matrix $\mathbf{X}$, it
holds that
\[
\sum_{n=0}^{+\infty}\Phi_{\mathbf{R}}^{(n)}(\mathbf{X}\mathbf{X}^{\ast}
)\leq\Vert\mathbf{X}\Vert^{2}\sum_{n=0}^{+\infty}\Phi_{\mathbf{R}}
^{(n)}(\mathbf{I})\leq C(z)\,\Vert\mathbf{X}\Vert^{2}\,\mathbf{I}<+\infty
\]
In order to complete the proof of Proposition \ref{eq:convergence-series-Phi1}
, we just follow the proof of unicity of Proposition
\ref{proposition:existence_unicity}. We express $\mathbf{a}^{H}{\Phi}
_{1}(\mathbf{X})\mathbf{b}$ as
\[
\mathbf{a}^{H}{\Phi}_{1}^{(n)}(\mathbf{X})\mathbf{b}=\mathbf{a}^{H}{\Phi}
_{1}^{(n)}(\mathbf{X})\left[  {\Phi}_{\mathbf{T}^{H}}^{(n)}(\mathbf{I}
)\right]  ^{-1/2}\left[  {\Phi}_{\mathbf{T}^{H}}^{(n)}(\mathbf{I})\right]
^{1/2}\mathbf{b}
\]
use the Schwartz inequality as well as (\ref{eq:inequality-fundamental-Phi1}),
and obtain that
\[
|\mathbf{a}^{H}{\Phi}_{1}^{(n)}(\mathbf{X})\mathbf{b}|\leq\left(
\mathbf{a}^{H}{\Phi}_{\mathbf{R}}^{(n)}(\mathbf{X}\mathbf{X}^{\ast}
)\mathbf{a}\right)  ^{1/2}\left(  \mathbf{b}^{H}{\Phi}_{\mathbf{T^{H}}}
^{(n)}(\mathbf{I})\mathbf{b}\right)  ^{1/2}
\]
and that
\[
\sum_{n=0}^{+\infty}|\mathbf{a}^{H}{\Phi}_{1}^{(n)}(\mathbf{X})\mathbf{b}
|\leq\left(  \sum_{n=0}^{+\infty}\mathbf{a}^{H}{\Phi}_{\mathbf{R}}
^{(n)}(\mathbf{X}\mathbf{X}^{\ast})\mathbf{a}\right)  ^{1/2}\;\left(
\sum_{n=0}^{+\infty}\mathbf{b}^{H}{\Phi}_{\mathbf{T^{H}}}^{(n)}(\mathbf{I}
)\mathbf{b}\right)  ^{1/2}<+\infty
\]
as expected. \newline

We are now in position to complete the proof of
(\ref{eq:rate-convergence-R-T-towards-0}). For this, we consider a uniformly
bounded sequence of $ML\times ML$ matrices $\mathbf{A}_{N}$ and evaluate
$\frac{1}{ML}\mathrm{tr}\left(  \mathbf{A}_N(\mathbf{R}(z)-\mathbf{T}
(z))\right)  $. As previously, matrix ${\bf A}_N$ is denoted by ${\bf A}$ in order to short the notations. For this, we take Eq. (\ref{eq:equation-R-T}) as a starting
point, and assume that $z$ belongs to the set $E_{N}$ defined by
(\ref{eq:def-EN}). As $z\in E_{N}$, the series
\[
\sum_{n=0}^{+\infty}{\Phi}_{1}^{(n)}\left(  {\Phi}_{1}({\boldsymbol\Delta
})\right)
\]
is convergent and
\begin{equation}
\mathbf{R}-\mathbf{T}=\sum_{n=0}^{+\infty}{\Phi}_{1}^{(n)}\left(  {\Phi}
_{1}({\boldsymbol\Delta})\right)  \label{eq:resolution-equation-R-T}
\end{equation}
Therefore,
\[
\frac{1}{ML}\mathrm{tr}\left(  \mathbf{A}(\mathbf{R}-\mathbf{T})\right)
=\sum_{n=0}^{+\infty}\frac{1}{ML}\mathrm{tr}\left(  \mathbf{A}{\Phi}
_{1}^{(n+1)}({\boldsymbol\Delta})\right)
\]
or equivalently,
\[
\frac{1}{ML}\mathrm{tr}\left(  \mathbf{A}(\mathbf{R}-\mathbf{T})\right)
=\sum_{n=0}^{+\infty}\frac{1}{ML}\mathrm{tr}\left(  {\boldsymbol\Delta\Phi
}_{1}^{t(n+1)}(\mathbf{A})\right)
\]
where ${\Phi}_{1}^{t}$ is the operator defined by (\ref{eq:def-transpose-Phi1}).
The strategy of the proof consists in showing that the series 
$\sum_{n=0}^{\infty}{\Phi}_{1}^{t(n+1)}(\mathbf{A})$ is convergent, and that
\begin{equation}
\sup_{N}\left\Vert \sum_{n=0}^{\infty}{\Phi}_{1}^{t(n+1)}(\mathbf{A}
)\right\Vert <+\infty.\label{eq:bornitude-series}
\end{equation}
If (\ref{eq:bornitude-series}) holds, (\ref{eq:control-trace-Delta}) will
imply that
\[
\left\vert \frac{1}{ML}\mathrm{tr}\left(  \left(  \sum_{n=0}^{\infty}{\Phi
}_{1}^{t(n+1)}(\mathbf{A})\right)  {\boldsymbol\Delta}\right)  \right\vert
\leq C(z)\,\frac{L}{MN}
\]
if $z\in E_{N}$.

In order to establish the convergence of the series as well as
(\ref{eq:bornitude-series}), we use the following inequality: for each
$ML\times ML$ deterministic matrix $\mathbf{A}$, it holds that
\begin{equation}
{\Phi}_{1}^{t(n)}(\mathbf{A})\left[  {\Phi}_{\mathbf{T}}^{t(n)}(\mathbf{I)}
\right]  ^{-1}\left(  {\Phi}_{1}^{t(n)}(\mathbf{A})\right)  ^{H}\leq{\Phi
}_{\mathbf{R}^{H}}^{t(n)}(\mathbf{A}^{H}\mathbf{A})\label{eq:inequality-Phi1t}
\end{equation}
where the operators ${\Phi}_{\mathbf{T}}^{t}$ and ${\Phi}_{\mathbf{R}^{H}}
^{t}$ are defined by
\begin{equation}
{\Phi}_{\mathbf{T}}^{t}(\mathbf{X})=c_{N}\,|z|^{2}\,{\Psi}\left(
\widetilde{\mathbf{T}}^{T}\overline{{\boldsymbol\Psi}}(\mathbf{T}
\mathbf{X}\mathbf{T}^{H})\widetilde{\mathbf{T}}^{\ast}\right)  \;,{\Phi
}_{\mathbf{R}^{H}}^{t}(\mathbf{X})=c_{N}\,|z|^{2}\,{\Psi}\left(
\widetilde{\mathbf{R}}^{\ast}\overline{{\Psi}}(\mathbf{R}^{H}\mathbf{X}
\mathbf{R})\widetilde{\mathbf{R}}^{T}\right)  \label{eq:def-Phi-R-T-t}
\end{equation}
The proof is similar to the proof of (\ref{eq:inequality-fundamental-Phi1}),
and is thus omitted. We remark that operators ${\Phi}_{\mathbf{T}}$ and
${\Phi}_{\mathbf{T}}^{t}$ are linked by the formula:
\begin{equation}
{\Phi}_{\mathbf{T}}^{t}(\mathbf{X})=\mathbf{T}^{-1} {\Phi
}_{\mathbf{T}}(\mathbf{T}\mathbf{X}\mathbf{T}^{H})\mathbf{T}^{-H}
\label{eq:link-PhiT-PhiTt}
\end{equation}
which implies that
\begin{equation}
{\Phi}_{\mathbf{T}}^{t(n)}(\mathbf{X})=\mathbf{T}^{-1} {\Phi
}_{\mathbf{T}}^{(n)}(\mathbf{T}\mathbf{X}\mathbf{T}^{H})\mathbf{T}
^{-H}\label{eq:link-PhiTn-PhiTtn}
\end{equation}
Since $\sum_{n=0}^{+\infty} {\Phi}_{\mathbf{T}}^{(n)}(\mathbf{I}
)<+\infty$, it holds that
\[
\sum_{n=0}^{+\infty} {\Phi}_{\mathbf{T}}^{t(n)}(\mathbf{I})\leq
\Vert\mathbf{T}^{-1}\Vert^{2}\Vert\mathbf{T}\Vert^{2}\sum_{n=0}^{+\infty
} {\Phi}_{\mathbf{T}}^{(n)}(\mathbf{I})<+\infty
\]
Moreover, we have already shown that $\sum_{n=0}^{+\infty} {\Phi
}_{\mathbf{T}}^{(n)}(\mathbf{I})\leq C(z)\;\mathbf{I}$, and that
$\mathbf{T}\mathbf{T}^{H}\geq\frac{\delta_{z}^{2}}{Q_{2}(|z|)}\mathbf{I}$, or
equivalently that $\Vert\mathbf{T}^{-1}\Vert^{2}\leq C(z)$. Therefore, it
holds that
\begin{equation}
\sum_{n=0}^{+\infty} {\Phi}_{\mathbf{T}}^{t(n)}(\mathbf{I})\leq
C(z)\;\mathbf{I.}\label{eq:borne-sup-sum-phiTtn}
\end{equation}
In order to control the series $\sum_{n=0}^{+\infty} {\Phi}
_{\mathbf{R}^{H}}^{(n)}(\mathbf{A}^{H}\mathbf{A})$, we remark that
\begin{equation}
{\Phi}_{\mathbf{R}^{H}}^{t}(\mathbf{X})=\mathbf{R}^{-H} {\Phi
}_{\mathbf{R}^{H}}(\mathbf{R}^{H}\mathbf{X}\mathbf{R})\mathbf{R}
^{-1}\label{eq:link-PhiR-PhiRt}
\end{equation}
which implies that
\begin{equation}
 {\Phi}_{\mathbf{R}^{H}}^{t(n)}(\mathbf{X})=\mathbf{R}^{-H} {\Phi
}_{\mathbf{R}^{H}}^{(n)}(\mathbf{R}^{H}\mathbf{X}\mathbf{R})\mathbf{R}
^{-1}\label{eq:link-PhiRn-PhiRtn}
\end{equation}
Using the same kind of arguments as above, we obtain that
\begin{equation}
\sum_{n=0}^{+\infty} {\Phi}_{\mathbf{R}^{H}}^{t(n)}(\mathbf{A}^{\ast
}\mathbf{A})\leq C(z)\;\mathbf{I}\label{eq:borne-sup-sum-phiRtn}
\end{equation}
if $z$ belongs to a set $E_{N}$ defined by (\ref{eq:def-EN}). Noting that
\[
\left\vert \mathbf{a}^{H}\sum_{n=0}^{\infty}{\Phi}_{1}^{t(n)}(\mathbf{A}
)\mathbf{b}\right\vert \leq\left[  \mathbf{a}^{H}\left(  \sum_{n=0}^{+\infty
} {\Phi}_{\mathbf{T}}^{t(n)}(\mathbf{I})\right)  \mathbf{a}\right]
^{1/2}\;\left[  \mathbf{b}^{H}\left(  \sum_{n=0}^{+\infty} {\Phi
}_{\mathbf{R}^{H}}^{t(n)}(\mathbf{A}^{\ast}\mathbf{A})\right)  \mathbf{b}
\right]  ^{1/2}
\]
we obtain that
\[
\Vert\sum_{n=0}^{\infty}{\Phi}_{1}^{t(n)}(\mathbf{A})\Vert\leq C(z)
\]
as soon as $z\in E_{N}$. This completes the proof of
(\ref{eq:rate-convergence-R-T-towards-0}).

\end{document}